\documentclass[12pt]{amsart}
\usepackage[margin=1in]{geometry}
\usepackage{amsmath, amssymb, amsthm,hyperref}
\usepackage{mathtools,xcolor}
\usepackage{enumitem}
\usepackage{multicol}
\usepackage{bbm} 

\usepackage{soul}

\newtheorem{theorem}{Theorem}[section]
\newtheorem{lemma}[theorem]{Lemma}
\newtheorem{proposition}[theorem]{Proposition}
\newtheorem{corollary}[theorem]{Corollary}
\theoremstyle{definition}
\newtheorem{definition}[theorem]{Definition}
\newtheorem{example}[theorem]{Example}
\newtheorem{question}[theorem]{Question}

\newtheorem{remark}[theorem]{Remark}

\newcommand{\kk}{\bold k}

\newcommand{\g}{\mathfrak{g}}

\numberwithin{equation}{section} 
\usepackage{color}

\begin{document}

\title[$p$-curvature of periodic pencils of flat connections]{$p$-curvature of periodic pencils of flat connections}

\author{Pavel Etingof}						

\address{Department of Mathematics, MIT, Cambridge, MA 02139, USA}

\author{Alexander Varchenko}

\address{Department of Mathematics, University of North Carolina at Chapel Hill, 
CB\# 3250 Phillips Hall
Chapel Hill, N.C. 27599, USA} 

\maketitle

\centerline{\bf To Boris Feigin on his 70-th birthday with admiration} 

\begin{abstract} In \cite{EV} we introduced the notion of a periodic pencil of flat connections on a smooth variety $X$. Namely, a pencil is a linear family of flat connections  $\nabla(s_1,...,s_n)=d-\sum_{i=1}^r\sum_{j=1}^ns_jB_{ij}dx_i,$
 where $\lbrace x_i\rbrace$ are coordinates on $X$ and $B_{ij}: X\to {\rm Mat}_N$ are matrix-valued regular functions. A pencil is periodic if it is generically invariant under the shifts $s_j\mapsto s_j+1$ up to isomorphism. In this paper we show that in characteristic $p>0$, the $p$-curvature operators $\lbrace C_i,1\le i\le r\rbrace$ of a periodic pencil $\nabla$ are isospectral to the commuting endomorphisms $C_i^*:=\sum_{j=1}^n (s_j-s_j^p)B_{ij}^{(1)}$, where $B_{ij}^{(1)}$ is the Frobenius twist of $B_{ij}$. Using the results of \cite{EV}, this allows us to compute the eigenvalues of the $p$-curvature for many important examples of pencils of flat connections, including Knizhnik-Zamolodchikov (KZ), Casimir, and Dunkl connections, their confluent limits, and equivariant quantum connections for conical symplectic resolutions with finitely many torus fixed points. In particular, for rational values of parameters these eigenvalues are zero, so the connections are globally nilpotent. We also show that every periodic pencil has regular singularities and its residues have rational eigenvalues for rational values of parameters. In particular, this holds for the aforementioned quantum connections if they have rational coefficients. Also we generalize these results to irregular pencils (KZ, Casimir, Dunkl, and Toda), and relate them in the Dunkl case to representations of rational Cherednik algebras. Finally, we extend our main result to pseudo-pencils and discuss the generalization to difference equations. 
\end{abstract} 

\tableofcontents

\section{Introduction}

Grothendieck's $p$-{\bf curvature} is an important invariant of linear ODE on algebraic curves and, more generally, of holonomic systems of linear PDE on algebraic varieties. For example, if such a system over $\Bbb C$ has an algebraic fundamental solution then its reduction at almost all primes $p$ has zero $p$-curvature, and the converse is the Grothendieck-Katz conjecture. Further, if such an irreducible system is geometric (i.e., occurs as a composition factor in a Gauss-Manin connection on the cohomology of fibers of a smooth morphism) then by a theorem of N. Katz it is {\bf globally nilpotent} (i.e., its reduction at almost all primes $p$ has nilpotent $p$-curvature), and the converse is the Andr\'e-Bombieri-Dwork conjecture. 

However, for a general holonomic system the $p$-curvature (and even its spectrum) is hard to compute, even if $X$ is an open subset of $\Bbb A^1$. The goal of this paper is to show that nevertheless for many important holonomic systems the $p$-curvature can be studied effectively, and in particular one can compute
its spectrum.  

Namely, in \cite{EV} we introduced the notion of a {\bf periodic pencil of flat connections}\footnote{ In algebraic geometry, the term ``pencil" is normally used for 1-parameter linear families, rather than multi-parameter ones, which is a slight mismatch of terminology. Still, we will use the term ``pencil" to be consistent with our previous paper \cite{EV}.} and showed that many important connections, such as Knizhnik-Zamolodchikov (KZ), Casimir, and Dunkl connections, as well as equivariant quantum connections of conical symplectic resolutions with finitely many torus fixed points fall into this category.\footnote{To realize the equivariant quantum connection of a conical symplectic resolution with finitely many torus fixed points as a periodic pencil, one may use the basis of cohomological stable envelopes of \cite{MO} for the equivariant cohomology. Also, since such connections are  generally not known to converge and have rational coefficients, in general one has to work with formal connections, i.e., ones whose coefficients are formal series in the Novikov variables. However, our main results extend mutatis mutandis to the formal case.} Our main motivation in introducing this notion was the computation of the spectrum of the $p$-curvature of such connections in positive characteristic, which is what we do in this paper. Namely, our main result (Theorem \ref{main1}) states that the $p$-curvature operators 
$\lbrace C_i, 1\le i\le r\rbrace$ of a periodic pencil 
$\nabla(\bold s)=d-\sum_{i=1}^r\sum_{j=1}^n s_jB_{ij}dx_i$ are isospectral to
the commuting endomorphisms
 $\sum_{i=1}^r\sum_{j=1}^n (s_j-s_j^p)B_{ij}^{(1)}dx_i$, where $x_i$ are local coordinates on $X$, $B_{ij}: X\to {\rm Mat}_N$ are regular functions, and $B_{ij}^{(1)}$ is the Frobenius twist of $B_{ij}$. In particular, it follows that periodic pencils over $\overline{\Bbb Q}$ are globally nilpotent for $\bold s\in \Bbb Q^n$. Using this theorem, we compute the spectrum of the $p$-curvature for the above examples of periodic pencils and show (using a theorem of Katz) that every periodic pencil over $\Bbb C$ has regular singularities and its residues have rational eigenvalues for $\bold s\in \Bbb Q^n$. Thus, 
as a by-product we obtain an interesting corollary in enumerative geometry: if the quantum connection of a conical symplectic resolution with finitely many torus fixed points has rational coefficients, then it has regular singularities and its residues have rational eigenvalues for rational values of parameters.

Theorem \ref{main1} also has an application to symplectic geometry. Namely, it is shown in \cite{Lee} that if $\mathcal X$ is a conical symplectic resolution with finitely many fixed points under the action of a torus $\bold T$,
 then the equivariant version of Fukaya's quantum Steenrod operation ${\rm St}$ on $H^*_{\bold T}(\mathcal X,\Bbb F_p)$ coincides with the $p$-curvature $C$ of the equivariant quantum connection $\nabla(\bold s)$ for $\mathcal X$, provided the latter has simple spectrum (and conjecturally always). In any case, it is proved in \cite{Lee} that these two operators commute and coincide after raising to power $p^i$ for some $i$, so in particular they are isospectral. On the other hand, the pencil $\nabla(\bold s)$ is periodic due to existence of geometric shift operators (see \cite{BMO,MO} and references therein), so our result yields the spectrum of $C$, and thereby the spectrum of ${\rm St}$.    
 
The organization of the paper is as follows. 

In Section 2 we discuss preliminaries. 

In Section 3 we prove Theorem \ref{main1}. We then apply this theorem to show that periodic pencils are globally nilpotent at rational parameter values and have regular singularities and its residues have rational eigenvalues for rational values of parameters. Further, we introduce infinitesimally-periodic and mixed-periodic pencils, which arise as confluent limits of periodic pencils, and generalize Theorem \ref{main1} to such pencils. Next, we describe applications of Theorem \ref{main1} to periodic pencils from \cite{EV}.  Finally, we discuss the connection of our results with Katz's theorem on global nilpotence of Gauss-Manin connections, with the Andr\'e-Bombieri-Dwork conjecture, and deduce regularity of quantum connections for symplectic resolutions. 

In Section 4 we discuss irregular pencils (KZ, Casimir, Dunkl, and Toda) and compute the spectrum of their $p$-curvature. In particular, in the case of Dunkl connections, we discuss the relation of their $p$-curvature with representation theory of rational Cherednik algebras in positive characteristic. 

In Section 5, we generalize Theorem \ref{main1} to pseudo-pencils, 
which satisfy the pencil condition (homogeneous linearity in $\bold s$) only locally
near $\infty$ and $0$, up to a gauge transformation. This allows 
us to include many more examples, in particular some examples 
of generalized Gauss-Manin connections. 

Finally, in Section 6 we discuss generalization of the theory 
of periodic families of flat connections and their $p$-curvature to  
difference and $q$-difference connections.\footnote{The first version of this paper appeared in January 2024. Subsections 3.8, 4.5 and Sections 5, 6 were added in December 2024.}
 
{\bf Acknowledgements.} The authors thank Vadim Vologodsky, Peter Koroteev, Jae Hee Lee and Andrey Smirnov for useful discussions and anonymous referees for helpful comments. P. E.'s work was partially supported by the NSF grant DMS-2001318 and  A. V.'s work was partially supported by the NSF grant DMS-1954266. 

\section{Preliminaries} 
Throughout the paper,  $\kk$ denotes an algebraically closed field, $V$ a finite dimensional $\kk$-vector space, and   
$X$ a smooth irreducible algebraic variety over $\kk$, unless specified otherwise.
 
\subsection{Reduction to characteristic $p$} 

Let $Y$ be any algebro-geometric structure defined over $\overline{\Bbb Q}$ which depends on finitely many parameters (e.g., variety, morphism of varieties, vector bundle, connection, etc.). Then $Y$ can be defined over a finitely generated subring $R\subset \overline{\Bbb Q}$. More precisely, there exists a form $Y_R$ of $Y$ defined over $R$ such that $Y=Y_R\otimes_R\overline{\Bbb Q}$. For every prime $p$, there are finitely many homomorphisms $\phi: R\to \overline{\Bbb F}_p$ (at least one for almost all $p$). Given such $\phi$, we can define the corresponding {\bf reduction of $Y$ at $p$}, ${\rm Red}_p(R,Y_R,\phi):=Y_R\otimes_{R,\phi} \overline{\Bbb F}_p$. { Of course, the reductions attached to two different choices of $(R,Y_R,\phi)$ may or may not be isomorphic.} But it is easy to show that for any two choices $(R',Y_{R'})$ and $(R'',Y_{R''})$ the collection of reductions corresponding to various $\phi$ is the same for almost all $p$. Thus it makes sense to say that reductions of $Y$ at almost all $p$ enjoy a certain property. 
 
\subsection{Frobenius twists and morphisms}   
Let $R,R'$ be commutative rings and $\phi: R\to R'$ be a homomorphism; 
it defines the corresponding map of spectra ${\rm Spec}\phi: {\rm Spec}R'\to {\rm Spec}R$. 
If $M$ is an $R$-module then we can define its pushforward
$\phi_*M:=R'\otimes_{R,\phi} M$, an $R'$-module. 
This defines the pullback functor $({\rm Spec}\phi)^*: {\rm QCoh}({\rm Spec}R)\to {\rm QCoh}({\rm Spec}R')$ on quasi-coherent sheaves. 

Let $R$ be a commutative ring of characteristic $p$ (i.e., an $\Bbb F_p$-algebra). 
Then $R$ has the {\bf absolute Frobenius endomorphism} ${\bf Fr}_R: R\to R$ 
given by ${\bf Fr}_R(a)=a^p$, $a\in R$. If $M$ is an $R$-module then its {\bf Frobenius twist} is the $R$-module $M^{(1)}:=({\bf Fr}_R)_*M$. In other words, speaking geometrically, if $S={\rm Spec}R$ is an affine scheme over $\Bbb F_p$ then we have the absolute Frobenius endomorphism of $S$, ${\bf Fr}_S={\rm Spec}{\bf Fr}_R: S\to S$, and if $M$ is a quasi-coherent sheaf on $S$ then its Frobenius twist is $M^{(1)}:={\bf Fr}_S^* M$.  

Now suppose that $X\to S$ is an affine scheme over $S$. Then the absolute Frobenius ${\bf Fr}_X: X\to X$ is in general not a morphism of $S$-schemes, since the corresponding map on regular functions is not $R$-linear but rather twisted $R$-linear (via ${\bf Fr}_R$).  To correct this, 
define the {\bf Frobenius twist} $X^{(1)}$ of $X$ to be the fiber product of $X$ with $S$ 
over $S$ twisted by ${\bf Fr}_S$; i.e., $\mathcal O(X^{(1)})=\mathcal O(X)^{(1)}$ as 
an $\mathcal O(S)$-algebra. Clearly, this is a functor on the category of $S$-schemes. 
Then the {\bf relative Frobenius morphism} is the natural map of $S$-schemes ${\rm Fr}_X: X\to X^{(1)}$. 

We will consider this setup in the case when $R=\kk$ is an algebraically closed field 
of characteristic $p$ and $S={\rm Spec}\kk$, so ${\bf Fr}_R={\bf Fr}_\kk$ is an automorphism. In this case, the Frobenius twist functor on $R$-modules (i.e., $\kk$-vector spaces) 
has a nice alternative description. Namely, for a $\kk$-vector space $V$, 
the space $V^{(1)}$ is the subspace of ${\rm Sym}^pV$ of elements $v^p$, $v\in V$.
It is clear that if $\lbrace v_i\rbrace$ is a basis of $V$ then $\lbrace v_i^p\rbrace$ is a basis of $V^{(1)}$, and that the assignment $V\mapsto V^{(1)}$ is 
an additive symmetric monoidal functor. Namely, its action on morphisms is as follows: 
for a linear map $T: V\to W$ with matrix $(t_{ij})$ 
in bases $\lbrace v_i\rbrace$ of $V$, $\lbrace w_j\rbrace$ of $W$, the matrix 
of $T^{(1)}$ in the bases $\lbrace v_i^p\rbrace$ of $V^{(1)}$, $\lbrace w_j^p\rbrace$ of $W^{(1)}$ is $(t_{ij}^p)$. Thus if $V=W$ are finite dimensional then the eigenvalues of $T^{(1)}$ are $\lambda_j^p$, where $\lambda_j$ are the eigenvalues of $T$.

Note that in this case the relative Frobenius morphism is the natural twisted-linear isomorphism $V\cong V^{(1)}$ which sends $v\in V$ to $v^p$. Thus we may alternatively define $V^{(1)}$ as the $\Bbb F_p$-space $V$ with twisted scalar multiplication given by $\lambda\cdot v:=\lambda^{1/p}v$, $\lambda\in \kk$. In this realization, $T^{(1)}$ 
is identified with $T$. 

Suppose now that $X$ is an affine $\kk$-scheme. 
Thus $X^{(1)}=X$ as an $\Bbb F_p$-scheme but with $\kk$-structure twisted by the map $\lambda\mapsto \lambda^{1/p}$. As explained above, the absolute Frobenius gives rise to the relative Frobenius morphism of $\kk$-schemes ${\rm Fr}_X: X\to X^{(1)}$, which on regular functions is defined by the formula $f\mapsto f^p$. Note that ${\rm Fr}_X$ is a homeomorphism in Zariski topology, so we may identify $X$ with $X^{(1)}$ as topological spaces using ${\rm Fr}_X$. 

Finally, note that by considering affine open covers, these definitions extend straightforwardly to not necessarily affine schemes.

\subsection{Isospectrality} 
Let $R$ be a finitely generated commutative $\kk$-algebra. Let $V$ be an $R$-module which is finite dimensional over $\kk$. Then $V$ has a finite composition series consisting of simple $R$-modules, which must be of the form $R/\mathfrak{m}\cong \kk$, where $\mathfrak{m}\subset R$ is a maximal ideal. Thus we obtain a function $\mu_V: {\rm Specm}R\to \Bbb Z_{\ge 0}$ with finite support such that $\mu_V(\mathfrak{m})$ is the multiplicity of $R/\mathfrak{m}$ in the Jordan-H\"older series of $V$. This function is called the {\bf spectral multiplicity} of $V$, and its support is called the {\bf spectrum of $V$}, denoted $\Sigma(V)$. If $\mu_V|_{\Sigma(V)}=1$, we say that 
$V$ has {\bf simple spectrum}. It is easy to see that if $V$ has simple spectrum then it is semisimple (a direct sum of simple modules). 

Two finite dimensional $R$-modules $V,W$ having the same spectral multiplicity are called {\bf isospectral}.
Let $V_s$ be the semisimplification of $V$, i.e., the direct sum of all simple composition factors of $V$. Thus $V,W$ are isospectral iff $V_s\cong W_s$. 

Now let $\bold U$ be a finite dimensional $\kk$-vector space and $R:={\rm Sym}\bold U$. 
Then an $R$-module is just a $\kk$-vector space $V$
equipped with a linear map $L: \bold U\to {\rm End}V$
such that 
$$[L(\bold u_1),L(\bold u_2)]=0$$
for all $\bold u_1,\bold u_2\in \bold U$. Thus 
if $\bold u=(u_1,...,u_r)$ in some basis of $\bold U$ 
then $L(\bold u)=\sum_{i=1}^r u_iL_i$, where 
$L_1,...,L_r\in {\rm End}V$ are commuting linear operators. 
We call such $L$ a {\bf pencil of commuting endomorphisms}, and say that 
$V$ corresponds to $L$, writing $V=V_L$. 
We say that two such pencils $L,M$ are {\bf isospectral} if 
so are the modules $V_L$ and $V_M$.  

\begin{lemma}\label{isospeccrit} 
$L,M$ are isospectral if and only if for every $\bold u\in \bold U$, the characteristic polynomials of the operators $L(\bold u)$ and $M(\bold u)$ (or, equivalently, their eigenvalues counting multiplicities) are the same. 
\end{lemma} 

\begin{proof} The ``only if" part is obvious, so let us prove the ``if" part. Pick $\bold u\in \bold U$ which separates points of $\Sigma(V_L)\cup \Sigma(V_M)\subset \bold U^*$. This is possible, since any element of $R=\kk[u_1,...,u_r]$ vanishing on $\kk^r$ is zero. For any $\ell\in \Sigma(V_L)\cup \Sigma(V_M)$, $\mu_{V_L}(\ell)$, respectively $\mu_{V_M}(\ell)$, equals the multiplicity of the eigenvalue $\ell(\bold u)$ for the operator $L(\bold u)$, respectively $M(\bold u)$. Thus for such $\ell$, we have $\mu_{V_L}(\ell)=\mu_{V_M}(\ell)$, hence $\mu_{V_L}=\mu_{V_M}$, as desired. 
\end{proof}

\begin{example}\label{twoco} If ${\rm char}(\kk)=p$ and $L:=\sum_{i=1}^r u_iL_i: \bold U\to {\rm End}V$ is a pencil of commuting endomorphisms, then the pencil $L(\bold u)^p=\sum_{i=1}^r u_i^{(1)}L_i^p\in {\rm End}V$ parametrized by $\bold U^{(1)}$ (where $u_i^{(1)}:=u_i^p$) is isospectral to the pencil $L(\bold u)^{(1)}=\sum_{i=1}^r u_i^{(1)}L_i^{(1)}\in {\rm End} V^{(1)}$. Indeed, the eigenvalues of both are $\lambda_j(\bold u)^p$, where $\lambda_j(\bold u)$ are the eigenvalues of $L(\bold u)$. 
\end{example} 

\subsection{$p$-curvature} (\cite{K1}, Section 5) Let $\nabla$ be a flat connection on a vector bundle $\mathcal V$ over $X$. If $x_1,...,x_r$ are local coordinates on $X$ then we have commuting operators $\nabla_1,...,\nabla_r$ of covariant partial derivatives with respect to these coordinates, which act on the space $\Gamma_{\rm rat}(\mathcal V)$ of rational sections of $\mathcal V$ and determine $\nabla$.  

Suppose that ${\rm char}(\kk)=p>0$. Then the operators
$$
C_i=C_i(\nabla):=\nabla_i^p: \Gamma_{\rm rat}(\mathcal V)\to \Gamma_{\rm rat}(\mathcal V)
$$ 
are $\kk(X)$-linear and combine 
into a twisted 1-form 
$$
C=C(\nabla):=\sum_{i=1}^r C_i(\nabla)dx_i^{(1)}\in {\rm Fr}_X^*\Omega^1(X^{(1)})\otimes_{\mathcal O_X} {\rm End} \mathcal V,
$$
where ${\rm Fr}_X: X\to X^{(1)}$ is the relative Frobenius morphism and $x_i^{(1)}:=x_i^p$ are local coordinates on $X^{(1)}$. The form $C$ is independent on the choice of coordinates and is called the \linebreak {\bf $p$-curvature} of $\nabla$. Moreover, $[\nabla,C]=[C,C]=0$, i.e., $[\nabla_i, C_l] =[C_i,C_l]=0$ for all $i,l$.  Proofs of these and other properties of $p$-curvature can be found in \cite{K,K1}. 

\begin{example}\label{pcurfor} If $X=\Bbb A^1$ and $\nabla=d+a$, $a\in {\rm Mat}_N(\kk[x])$,  
then for $p=2$ we have $C=a^2+a'$ and for $p=3$ we have 
$C=a^3+[a',a]+a''$ (where we identify 1-forms with functions using the coordinate $x$).  
\end{example} 

An important application of $p$-curvature is the computation of the space of local flat sections of $\nabla$ (see e.g. \cite{K}). Let us recall this computation. For $\bold x\in X(\kk)$ let $\mathcal K_{\bold x}\subset \mathcal V_{\bold x}$ be the kernel of the $p$-curvature $C(\nabla)(\bold x)$. 
Assume that $\dim \mathcal K_{\bold x}=d$ for all $\bold x$
(this can be achieved by replacing $X$ by a dense open subvariety). 
Thus we have a vector bundle $\mathcal K$ on $X$ whose fibers are 
$\mathcal K_{\bold x}$. Let $\mathcal S$ be the sheaf of flat sections of $\nabla$. 

\begin{proposition}\label{solu} The sheaf $\mathcal S$ is a locally free 
coherent sheaf of rank $d$ over $X^{(1)}$, and $\mathcal K={\rm Fr}_X^*\mathcal S$. 
\end{proposition} 

\begin{proof} Since $\nabla$ commutes 
with $C$, it defines a flat connection on $\mathcal K$. Moreover, if  
$F$ is a flat section of $\nabla$, i.e., $\nabla_iF=0$ 
for all $i$, then $C_iF=\nabla_i^pF=0$, so 
$F$ is a section of $\mathcal K$. 
Thus, restricting attention to $\mathcal K$, we 
may assume without loss of generality that $C=0$.

Let $\bold x\in X(\kk)$. On some affine neighborhood $U$ of $\bold x$ we can choose local coordinates $x_1,...,x_r$ near $\bold x$ and identify $\mathcal V$ with the trivial bundle $\kk^N\times U\to U$. Our job is to show that  
for any $F_0\in \kk^N$ there exists a flat section 
$F(x_1,...,x_r)$ of $\nabla$ such that $F(0)=F_0$. 
But one such section is given by the formula
$$
F=(-1)^r \prod_{i=1}^r
\nabla_i^{p-1} (F_0\prod_{i=1}^r x_i^{p-1}),
$$
which completes the proof.\footnote{Recall that the Frobenius pullback of any vector bundle (in particular, ${\rm Fr}_X^*\mathcal S$) has a canonical flat connection with zero $p$-curvature. The isomorphism ${\rm Fr}_X^*\mathcal S\cong \mathcal K$ maps this connection to $\nabla$.}
\end{proof} 
 
More concretely, Proposition \ref{solu} implies that if $\mathcal K_{\bold y}$ 
has constant dimension $d$ for $\bold y$ in a neighborhood of $\bold x\in X(\kk)$
then there exists an affine neighborhood $U$ of $\bold x$ and 
flat sections $F_1,...,F_d$ of $\nabla$ over $U$ 
such that $F_1(\bold y),...,F_d(\bold y)$ is a basis 
of $\mathcal K_{\bold y}$ for all $\bold y\in U(\kk)$, and 
$F_1,...,F_d$ is a basis of $\mathcal S(U)$ over $\mathcal O(U^{(1)})$. 
 
\begin{remark} Suppose $\nabla$ is a rational connection on $\Bbb A^1$ defined over $\overline{\Bbb Q}$. If $\nabla$ has a basis of algebraic solutions (which for regular connections is equivalent to having finite monodromy), then by Proposition \ref{solu} the $p$-curvature of the reduction of $\nabla$ to characteristic $p$ is zero for almost all $p$, since the algebraic solutions can be reduced modulo large $p$. The converse statement is known as the {\bf Grothendieck-Katz conjecture}, which first appeared in \cite{K} and is still open. 
\end{remark} 
 
\section{Periodic pencils of flat connections and their $p$-curvature}

\subsection{Periodic families and pencils of flat connections} 
In \cite{EV} we defined periodic families of flat connections. 
Let us recall this definition.  

\begin{definition} An {\bf $n$-parameter family of flat connections} on $X$ with values in $V$ is a family of flat connections $\nabla(\bold s)=d-B(\bold s)$, 
$\bold s:=(s_1,...,s_n)$, on the trivial vector bundle\footnote{ As explained in \cite{EV}, Remark 3.2, the notion of a pencil is not gauge-invariant, so it only makes sense with respect to a fixed trivialization of the vector bundle, and is not well defined for a general vector bundle. The notion that {\it is} gauge invariant and makes sense in general is that of an {\bf affine pencil}, i.e., a family of flat connections of the form $\nabla_0-s_1B_1-...-s_nB_n$ where $\nabla_0$ is a fixed connection, which in a trivialization looks like $d-B$ where $B=B_0+s_1B_1+...+s_nB_n$, an {\it affine} linear function of $s_i$.}
 $X\times V\to X$, 
where $B\in \Omega^1(X)\otimes {\rm End V}[\bold s]$. The family $\nabla$ 
is said to be a {\bf pencil} if $B=\sum_{j=1}^n s_jB_j$, $B_j\in \Omega^1(X)\otimes {\rm End}V$. 
\end{definition} 

Thus a pencil of flat connections is determined by a collection   
of 1-forms $B_j$ on $X$ with values in ${\rm End} V$
such that $dB_j=[B_j,B_k]=0,\ 1\le j,k\le n$ (these conditions are vacuous if $\dim X=1$). 

\begin{definition} A family $\nabla$ is said to be {\bf periodic} if 
there exist {\bf shift operators} 
$$
A_j\in GL(V)(\kk(\bold s)[X]),\ 1\le j\le n
$$
 such that 
$$
\nabla(\bold s+\bold e_j)\circ A_j(\bold s)=A_j(\bold s)\circ \nabla(\bold s),\ 1\le j\le n,
$$
where $\lbrace\bold e_j\rbrace$ is the standard basis of $\kk^n$. 
\end{definition} 

Suppose ${\rm char}(\kk)=0$ and $\nabla$ is a periodic family over $\kk$. In this case we can choose forms of $X$, $\nabla$ and $A$ over some finitely generated subring $S\subset \kk$. Then for any homomorphism $\phi: S\to \Bbb F$ from $S$ to a field $\Bbb F$ of characteristic $p$, the family $\nabla\otimes_S\Bbb F$ obtained from $\nabla$ by reduction to $\Bbb F$ via $\phi$ is a periodic family over $\Bbb F$. Thus every example of a periodic family in characteristic zero gives a periodic family
 in almost every positive characteristic.  

\subsection{$p$-curvature of periodic pencils} 
Now assume that ${\rm char}(\kk)=p$. Let $x_i,1\le i\le r$ be local coordinates on 
some dense open subset of $X$, $x_i^{(1)}:=x_i^p$ the corresponding coordinates on $X^{(1)}$, and $\partial_i:=\frac{\partial}{\partial x_i^{(1)}}$ the associated commuting vector fields on $X^{(1)}$. Let $\bold U\cong \kk^r$ be the $\kk$-vector space spanned by $\partial_i$, $1\le i\le r$.  Let $\nabla(\bold s)=d-\sum_{j=1}^n s_jB_j$ be a pencil of flat connections on $X$. Then for each $\bold s=(s_1,...,s_n)\in \kk^n$ 
we have two pencils of commuting endomorphisms: the $p$-curvature 
$$
C(\nabla(\bold s)): \bold U\to {\rm End} V
$$ 
and 
$$
\bold B(\nabla(\bold s)):=\sum_{j=1}^n (s_j-s_j^p)B_j^{(1)}=\sum_{i=1}^r\sum_{j=1}^n (s_j-s_j^p)B_{ij}^{(1)}dx_i^{(1)}: \bold U\to {\rm End} V^{(1)}.
$$
Our main result is the following theorem.

\begin{theorem}  \label{main1} If $\nabla(\bold s)$ is a periodic pencil then the pencils $C(\nabla(\bold s))$ and $\bold B(\nabla(\bold s))$ of commuting endomorphisms are 
isospectral for all $\bold s\in \kk^n$. In particular,
if $s_j\in \Bbb F_p$ for all $j$ then the endomorphisms $C_i(\nabla(\bold s))$ are nilpotent. 
Moreover, if for some $\bold s$, $\bold B(\nabla(\bold s))$ has simple spectrum (in the sense that $\bold B(\nabla(\bold s))(\bold x)$ has simple spectrum for generic $\bold x\in X$), then 
$C(\nabla(\bold s))$ and $\bold B(\nabla(\bold s))$ are conjugate.  
\end{theorem} 

\begin{proof} For $u_1,...,u_r\in \kk$ let 
$$
C(\bold u,\bold s):=\sum_{i=1}^r u_iC_i(\nabla(\bold s)),\ 
\bold B(\bold u,\bold s):=\sum_{i=1}^r\sum_{j=1}^n u_i(s_j-s_j^p)B_{ij}^{(1)}.
$$ 
By Lemma \ref{isospeccrit}, it suffices to show that for all $\bold u,\bold s$ the characteristic polynomials of the matrices
$C(\bold u,\bold s)$ and $\bold B(\bold u,\bold s)$ coincide. 

Fix $\bold u$ and let 
$$
b_m(\bold u,\bold s):={\rm Tr}\wedge^mC(\bold u,\bold s)
$$ 
be the $m$-th coefficient of the characteristic polynomial of $C(\bold u,\bold s)$ multiplied by $(-1)^m$. This is a polynomial of $\bold s$ of degree $pm$. The periodicity property of $\nabla$ implies that 
\begin{equation}\label{conju}
C(\nabla(\bold s+\bold e_j))\circ A_j(\bold s)=A_j(\bold s)\circ C(\nabla(\bold s)),
\end{equation} 
hence 
$b_m(\bold u,\bold s+\bold e_j)=b_m(\bold u,\bold s)$ 
for all $j$. Thus 
$$
b_m(\bold u,\bold s)=\beta_m(\bold u,\bold s-\bold s^p),
$$ 
where $\bold s^p:=(s_1^p,...,s_n^p)$ and 
$\beta_m$ is a polynomial of degree $m$ in the second variable.
 
Moreover, it is easy to see that the leading term of $b_m(\bold u,\bold s)$ in $\bold s$ is 
 $$
b_m^0(\bold u,\bold s)=(-1)^m{\rm Tr}\wedge^m\sum_{i=1}^r u_i(\sum_{j=1}^n s_jB_{ij})^p.
$$ 
Indeed, the $p$-curvature has the form 
$$
C(\bold u,\bold s)=-\sum_{i=1}^r u_i (\sum_{j=1}^n s_jB_{ij})^p+{\rm l.d.t.}
$$
where l.d.t. are the lower degree terms with respect to $\bold s$. 
In view of Example \ref{twoco}, we then have
$$
b_m^0(\bold u,\bold s)=(-1)^m{\rm Tr}\wedge^m\sum_{i=1}^r u_i\sum_{j=1}^n s_j^pB_{ij}^{(1)}.
$$ 
Hence the leading term $\beta_m^0$ of $\beta_m$ is 
$$
\beta_m^0(\bold u,\bold s)={\rm Tr}\wedge^m\sum_{i=1}^r\sum_{j=1}^n u_is_jB_{ij}^{(1)},
$$
so 
$$
\beta_m^0(\bold u,\bold s-\bold s^p)={\rm Tr}\wedge^m \bold B(\bold u,\bold s).
$$
Thus it remains to show that $\beta_m=\beta_m^0$, i.e., $\beta_m$
is homogeneous of degree $m$ in $\bold s$. To this end it suffices to show that the polynomial $\beta_m(\bold u,\bold s-\bold s^p)$ 
does not contain monomials of degree $<m$ in $\bold s$. But $C(\bold u,0)=0$, so 
all matrix coefficients of $C(\bold u,\bold s)$ belong to the ideal $(s_1,...,s_n)$. 
This implies that ${\rm Tr}\wedge^mC(\bold u,\bold s)$ does not contain monomials of degree $<m$, as desired. 
\end{proof} 

Thus the joint eigenvalues $\Lambda_{i\ell}$ of the operators $C_i(\nabla(\bold s))$ coincide with the joint eigenvalues of the operators $\bold B_i(\bold s):=\sum_{j=1}^n(s_j-s_j^p)B_{ij}^{(1)}$.  
If $n=1$ and $s_1=s$ then we have $\Lambda_{i\ell}=(s-s^p)\Lambda_{i\ell}^0$, where $\Lambda_{i\ell}^0$ are the joint eigenvalues of the operators $B_i^{(1)}$ which don't depend on $s$. However, for $n\ge 2$ these eigenvalues are more complicated algebraic (in general, irrational) functions of $\bold s$, since $B_{ij}$ for different $j$ in general don't commute. 

\begin{remark} If $\bold s\in \Bbb F_p^n$ then $\bold B(\nabla(\bold s))=0$ but the nilpotent 1-forms $C(\nabla(\bold s))$ need not be zero (except, of course, for $\bold s=0$). However, many of these forms are often 
conjugate to each other. For example, consider the case $n=1$,  
and let $\alpha_1,...,\alpha_\ell\in \Bbb F_p$ be the poles of $A,A^{-1}$ 
which belong to $\Bbb F_p$, arranged in ``increasing order", i.e., when $\alpha_j$ are regarded as elements of $[0,p-1]$, we have $0\le \alpha_1<\alpha_2<...<\alpha_\ell\le p-1$. Then $\Bbb F_p$ is divided into intervals 
$I_j:=[\alpha_j+1,\alpha_{j+1}]$, $j\in \Bbb Z/\ell$ (here the interval $I_0$ consists 
of such $m\in [0,p-1]$ that $\alpha_\ell+1\le m\le p-1$ or $0\le m\le \alpha_1$). Equation \eqref{conju} now implies 

\begin{proposition} For every $j\in \Bbb Z/\ell$, the forms $C(\nabla(m))$, $m\in I_j$ are conjugate to each other. In particular, if $m\in I_0$ then $C(\nabla(m))=C(\nabla(0))=0$.  
\end{proposition}  
\end{remark}  

\subsection{Infinitesimally periodic and mixed-periodic families} 

By replacing $\bold s$ with $\bold s/t$ in the definition of a periodic family 
of flat connections and sending $t$ to $0$, we obtain the definition 
of an infinitesimally periodic family. 

\begin{definition} A family of connections $\nabla(\bold s)=d-B(\bold s)$ on the trivial bundle on $X$ with fiber $V$ is {\bf infinitesimally periodic} if there exist {\bf infinitesimal shift operators} 
$$
a_j\in {\rm End}V\otimes \kk(\bold s)[X],\ 1\le j\le n
$$
such that for all $i,j$, 
$$
[\partial_{s_j}-a_j(\bold s,\bold x),\nabla_i(\bold s)]=0,
$$
i.e., $\widetilde \nabla(\bold s):=d-B(\bold s)-\sum_{j=1}^n a_j(\bold s)ds_j$ 
is a connection on the trivial bundle on $X\times \kk^n$ with fiber $V$ 
whose curvature is of the form\footnote{In fact, in all the examples we will consider this curvature vanishes. Moreover, if ${\rm char}(\kk)=0$ and $\nabla$ has no nontrivial endomorphisms over $\kk(\bold s)$, then it is easy to see that $\eta$ is a scalar-valued rational closed $2$-form independent on $\bold x$.} $\eta=\sum_{1\le i<j\le n}\eta_{ij}(\bold x,\bold s)ds_i\wedge ds_j$. 
\end{definition} 

\begin{lemma}\label{infper} Let ${\rm char}(\kk)=p$. Assume that $\nabla$ is an infinitesimally periodic family and $a_j$ are regular at $\bold s=0$. Then the matrix coefficients of the $p$-curvature operators $C_i(\bold s)$ of $\nabla(\bold s)$ belong to the ideal 
in $\kk[\bold s][X]$ generated by $s_1^p,...,s_n^p$. 
\end{lemma}  

\begin{proof} Since $[\partial_{s_j}-a_j(\bold s,\bold x),\nabla_i(\bold s)]=0$, we have  
\begin{equation} \label{eqqq} 
[\partial_{s_j}-a_j(\bold s,\bold x),C_i(\bold s,\bold x)]=0.
\end{equation}
Let $s_1^{i_1}...s_n^{i_n}$ be a minimal monomial occurring in the expansion of $C_i$ 
(i.e., the monomials obtained from this monomial by lowering the powers don't occur). 
Then \eqref{eqqq} implies that all $i_k$ must be divisible by $p$. 
Also $i_k$ cannot all be zero, as $C_i(0,\bold x)=0$. Thus all monomials occuring in $C_i$ 
are divisible by $s_j^p$ for some $j$. The lemma follows. 
\end{proof} 

More generally, one may consider {\bf mixed-periodic families} $\nabla(\bold s)$, 
$\bold s=(\bold s_1,\bold s_2)$, $\bold s_1\in \kk^{n_1}$, $\bold s_2\in \kk^{n_2}$, 
$n_1+n_2=n$, which are periodic in $\bold s_1$ for fixed $\bold s_2$ and infinitesimally periodic in $\bold s_2$ for fixed $\bold s_1$ (where $A_i$ and $a_j$ are defined over $\kk(\bold s)$). Interesting examples 
of such families arise as limits of periodic families when some of the parameters $s_j$ are replaced by $s_j/t$ and $t$ is sent to $0$, see Section 4 below. 

The following generalization of Theorem \ref{main1} to the case of mixed-periodic families is obtained by a straighforward generalization (or, more precisely, degeneration) of its proof. 
Let 
$$
\bold B(\nabla(\bold s)):=\sum_{j=1}^{n_1} (s_j-s_j^p)B_j^{(1)}-\sum_{j=n_1+1}^{n} s_j^pB_j^{(1)}.
$$

\begin{theorem}\label{main3} Let $\nabla(\bold s)$, $\bold s=(\bold s_1,\bold s_2)\in \kk^n$ be a mixed-periodic pencil on a smooth variety\footnote{We note that while Theorem \ref{main3} is formulated on a smooth variety $X$, 
it extends with the same proof to the case when $X$ is a formal polydisk, or a more general smooth formal scheme.} $X$ (periodic in $\bold s_1\in \kk^{n_1}$ and infinitesimally periodic in $\bold s_2\in \kk^{n_2}$). Assume that $a_j(\bold s_1,\bold s_2)$ are regular in $\bold s_1$ at $\bold s_1=0$ and $a_j(0,\bold s_2)$ are regular in $\bold s_2$ at $\bold s_2=0$. Then the pencils $C(\nabla(\bold s))$ and $\bold B(\nabla(\bold s))$ of commuting endomorphisms are isospectral for all $\bold s$. In particular, if $s_j\in \Bbb F_p$ for all $1\le j\le n_1$ and $s_j=0$ for $n_1+1\le j\le n$ then the endomorphisms $C_i(\nabla(\bold s_1,\bold s_2))$ are nilpotent. Moreover, if for some $\bold s$, $\bold B(\nabla(\bold s))$ has simple spectrum, then $C(\nabla(\bold s))$ and $\bold B(\nabla(\bold s))$ are conjugate. 
\end{theorem}

\begin{proof} Similarly to the proof of Theorem \ref{main1}, 
we see that 
$$
{\rm Tr}\wedge^m C(\bold u,\bold s)=\beta_m(\bold u,\bold s_1^p-\bold s_1,\bold s_2^p)
$$
where for each $\bold u$, $\beta_m(\bold u,-,-)$ is a polynomial of degree $m$; namely, the fact that this is a function of $\bold s_2^p$ follows from the Lax-type differential equations for the $p$-curvature
$$
\partial_{s_{j}}C(\bold s)=[a_j(\bold s),C(\bold s)],\ n_1+1\le j\le n,
$$
since if $b\in \kk[s]$ and $b'(s)=0$ then $b$ is a polynomial of $s^p$.   
So like in the proof of Theorem \ref{main1}, it suffices to show that the polynomial 
$\beta_m(\bold u,\bold s_1^p-\bold s_1,\bold t)$ 
contains no monomials of degree $<m$. To this end, it suffices to show that 
the matrix coefficients of $C(\bold u,\bold s_1,\bold s_2)$ 
belong to the ideal in $\kk[\bold s][X]$ 
generated by $s_1,...,s_{n_1},s_{n_1+1}^p,...,s_n^p$. For this it is enough to prove that 
the matrix coefficients of $C(\bold u, 0,\bold s_2)$ belong to the ideal 
in $\kk[\bold s][X]$ generated by $s_{n_1+1}^p,...,s_n^p$.
But this follows from Lemma \ref{infper}. 
\end{proof} 

\subsection{Global nilpotence and regularity of periodic pencils} 

\begin{definition}  (N. Katz, \cite{K,K1}) A flat connection $\nabla$ defined over $\overline{\Bbb Q}$ is said to be {\bf globally nilpotent} if the $p$-curvature of its reduction to characteristic $p$ is nilpotent for almost all $p$. 
\end{definition} 

\begin{corollary}\label{gnilp} If $\nabla$ is a periodic pencil defined over $\overline{\Bbb Q}$ then for any $\bold s\in \Bbb Q^n$, the connection $\nabla(\bold s)$ is globally nilpotent. 
\end{corollary} 

\begin{proof} This follows immediately from Theorem \ref{main1}. 
\end{proof} 

Let ${\rm char}(\kk)=0$. 
Recall that a family of flat connections $\nabla$ over $\kk$ is {\bf quasi-motivic} if it is periodic and has regular singularities (\cite{EV}, Definition 4.3). 

\begin{corollary}\label{regsin} Any periodic pencil $\nabla(\bold s)$ over $\kk$ has regular singularities and its residues have rational eigenvalues for $\bold s\in \Bbb Q^n$ (equivalently, quasiunipotent local monodromies for such $\bold s$ if $\kk=\Bbb C$). In particular, the notions of a periodic and a quasi-motivic family in the special case of pencils are equivalent. 
\end{corollary} 

\begin{proof} Without loss of generality we may assume that $\nabla$ is defined over $\overline{\Bbb Q}$. By a theorem of N. Katz conjectured by Grothendieck (\cite{K1}, Section 13), any globally nilpotent flat connection defined over $\overline{\Bbb Q}$ has regular singularities, and its residues have rational eigenvalues. Thus by Corollary \ref{gnilp} $\nabla(\bold s)$ has regular singularities for $\bold s\in \Bbb Q^n$ and its residues have rational eigenvalues. Pick a smooth compactification $\overline X$ of $X$ with normal crossing divisor $D$ at $\infty$. Since the singular locus of $\nabla(\bold s)$ on $\overline X$ is $D$ for all $\bold s\in \kk^n$ (i.e., they are not moving as $\bold s$ varies), it follows that the singularities of $\nabla(\bold s)$ are regular for all $\bold s\in \kk^n$.
\end{proof} 

\subsection{Applications of the main theorem} 

\begin{theorem}\label{main2} Let $\nabla(\bold s)$ 
be any of the pencils in \cite{EV}, Section 5 (KZ, Casimir,
 and Dunkl connections), or the equivariant quantum connection for a conical symplectic resolution with finitely many torus fixed points (\cite{EV}, Subsection 3.5). Then for almost all $p$, 
the $p$-curvature of a reduction $d-\sum_{j=1}^n s_jB_j$ of $\nabla(\bold s)$ at $p$
is isospectral to $\sum_{j=1}^n (s_j-s_j^p)B_j^{(1)}$. 
In particular, if $s_j\in \Bbb F_p$ for all $j$ then this $p$-curvature 
is nilpotent, so $\nabla(\bold s)$ is globally nilpotent.\footnote{We note that global nilpotence of the rational Dunkl connection $\nabla(\bold c)$ for a real reflection group $W$ with non-singular rational $\bold c$ (i.e., such that the corresponding Hecke algebra $H(e^{2\pi i\bold c})$ is semisimple) also follows from the recent paper \cite{EsG}. Indeed, in this case the monodromy representation of $\nabla(\bold c)$ is an (automatically semisimple) representation of $H(e^{2\pi i\bold c})$, so its simple summands are irreducible rigid local systems on the regular part $\mathfrak h_{\rm reg}$ of the reflection representation of $W$ with quasiunipotent monodromies at infinity. 
Thus we may apply \cite{EsG}, Theorem 1.2 and Corollary 1.3 to 
$X=\mathfrak h_{\rm reg}$ and $\overline X$ a smooth compactification of $X$ such that $D:=\overline X\setminus X$ is a normal crossing divisor. A similar argument works for complex reflection groups. (We are grateful to P. Godfard for pointing this out).}  
\end{theorem} 

\begin{proof} It is shown in \cite{EV} that these pencils are periodic, so 
the result immediately follows from Theorem \ref{main1}.
\end{proof} 

\begin{remark} More specifically, we see that Theorem \ref{main2} holds for those reductions of $\nabla$ to characteristic $p$ which come with reductions of the shift operators $A_j$. So if $A_j$ are known explicitly (or at least defined over an explicit finitely generated ring, as in the case of quantum connections where $A_j$ are constructed geometrically) then we can produce a list of forbidden primes away from  which Theorem \ref{main2} holds. However, in cases when the existence of $A_j$ is proved non-constructively (e.g., using that the pencil is regular and has periodic monodromy, as in the examples of \cite{EV}, Section 5) we cannot specify a lower bound for $p$ sufficient for Theorem \ref{main2}. 
\end{remark}

\begin{remark} Theorem \ref{main2} allows one to compute (albeit somewhat implicitly) the eigenvalues of the $p$-curvature in many of the above examples. For instance, for the KZ connection $d-\hbar\sum_{i\ne j}\frac{\Omega_{ij}}{x_i-x_j}dx_i$ on the tensor product of finite dimensional representations of a simple Lie algebra, Theorem \ref{main2} says that the $p$-curvature is isospectral to the Gaudin operators renormalized by the factor 
$\hbar-\hbar^p$, so its eigenvalues can be found using the 
Bethe Ansatz method.  
\end{remark} 

\subsection{Motivic families}\label{motfam}
Recall that a family $\nabla$ of flat connections on $X$ is said to be {\bf motivic} if there exists an irreducible variety $Y$, a smooth morphism $\pi: Y\to X$, and non-vanishing regular functions $\Phi_1,...,\Phi_n$ on $Y$ such that on some dense open set $X^\circ \subset X$, for Zariski generic $\bold s$ the connection $\nabla(\bold s)$ is isomorphic  to the generalized\footnote{We use the adjective ``generalized" to distinguish such connections from genuine Gauss-Manin connections on cohomology of $\pi^{-1}(\bold x)$ with trivial coefficients.} 
Gauss-Manin connection on $H^i(\pi^{-1}(\bold x),\mathcal L(\Phi_1,..,\Phi_n,\bold s))$ for some $i$, where $\mathcal L(\Phi_1,..,\Phi_n,\bold s)$ 
is the local system generated by the multivalued function 
$\prod_{j=1}^n \Phi_j^{s_j}$ (\cite{EV}, Definition 4.14). 
 
By \cite{EV}, Proposition 4.15, this is a subclass of quasi-motivic families. 

Also recall {\bf Katz's theorem}:

\begin{theorem}\label{katzth} (\cite{K}) The Gauss-Manin connection on the cohomology of fibers of a smooth morphism defined over $\overline{\Bbb Q}$ is globally nilpotent. 
\end{theorem} 

Katz's theorem implies that Corollary \ref{gnilp} holds for 
motivic families (not necessarily pencils):

\begin{proposition} If $\nabla$ is a motivic family then $\nabla(\bold s)$ 
is globally nilpotent for any $\bold s\in \Bbb Q^n$. 
\end{proposition}

\begin{proof} For $\bold s\in \Bbb Q^n$ the function $\prod_{j=1}^n \Phi_j^{s_j}$ can be viewed as a regular function on a finite cover $\widetilde Y$ of $Y$ whose degree is the common denominator of $s_j$, so the corresponding connection is the Gauss-Manin connection on the cohomology of the fibers of the map $\widetilde \pi: \widetilde Y\to X$. Thus the result follows by Katz's theorem. 
\end{proof} 

\begin{remark} There is another proof of Theorem \ref{main1} 
for 1-parameter motivic pencils $\nabla(s)$ attached to a multivalued function $\Phi^s$ on $Y$, provided that the ``master function" $\Phi$ has $\dim H^{\dim Y-\dim X}(\pi^{-1}(\bold x),\mathcal L(\Phi,s))$ nondegenerate critical points on fibers of $\pi: Y\to X$ (this happens, for instance, for the KZ connection for $\mathfrak{sl}_2$). This proof, proposed  by the second author and V. Vologodsky (\cite{VV}), is based on the Hodge-theoretic methods in the spirit of \cite{K}, showing that the $p$-curvature of $\nabla(s)$ reduced to characteristic $p$ 
``localizes" in a suitable sense to the critical points of $\Phi$. So this proof has the advantage of not using the linearity of $\nabla$ in $s$, and therefore extends to the more general case of motivic families (not just pencils) of flat connections. We hope that this method of proof can also be generalized to multiparameter families and to cases when critical points of $\Phi$ are degenerate, or even non-isolated.\footnote{{\bf Note added in December 2024}: In fact, recently such results have been announced in many examples of this sort (for quantum connections of Nakajima varieties) in the paper \cite{KS}.} 
\end{remark} 

\subsection{The Andr\'e-Bombieri-Dwork conjecture} 

The Andr\'e-Bombieri-Dwork conjecture \cite{A} states that if a first order system of linear differential equations 
on $\Bbb P^1$ defined over $\overline{\Bbb Q}$ is globally nilpotent, then it is geometric, i.e., its solutions admit an integral representation. While we can't offer any definitive ideas towards a proof or disproof of this conjecture, our results provide numerous examples of concrete globally nilpotent systems for which a geometric construction is not known (even though in some cases it has been sought for a long time). 

For instance, such examples include rational and trigonometric Dunkl connections for exceptional complex reflection groups (\cite{EV}, Subsection 5.3) and elliptic KZ connections (\cite{EV}, Subsection 5.1.7), for rational values of the parameters $s_j$. By Corollary \ref{gnilp}, they are all globally nilpotent, and can be made into connections on $\Bbb P^1$ by restricting to any line (or, more generally, rational curve) on the base. 

For instance, there is no known geometric construction of the (rational or trigonometric) Dunkl connection $\nabla(c)$ for exceptional groups $W=E_6,E_7,E_8$. In other words, there is no known integral representation for solutions of this system (Heckman-Opdam hypergeometric functions for $W$), even though considerable efforts have been made to find one. See e.g. the discussion at the beginning of \cite{CHL}. 

\subsection{Regularity of quantum connections of symplectic resolutions}

Let $\mathcal X$ be a conical symplectic resolution with finitely many torus fixed points. 
The quantum connection $\nabla_{\mathcal X}$ of $\mathcal X$ is a priori defined over the ring of formal series in the Novikov variables, but it is expected to converge 
and have rational coefficients, which is so in all examples where it has been computed. Also in all these examples $\nabla_{\mathcal X}$ happens to have regular singularities. The following corollary shows that 
the latter is, in fact, true in general under the rationality assumption.

\begin{corollary} If $\nabla_{\mathcal X}$ converges and has rational coefficients, then it has regular singularities and its residues have rational eigenvalues for rational values of parameters (equivalently, quasiunipotent local monodromies). 
\end{corollary} 

\begin{proof} Since $\nabla_{\mathcal X}$ is a periodic pencil (\cite{EV}, Subsection 3.5), this follows immediately from Corollary \ref{regsin}.
\end{proof} 

\section{Irregular pencils}
We would now like to apply the above results to computing the spectrum of the $p$-curvature for irregular connections. Many interesting examples of such connections arise as confluent limits $\nabla=\lim_{t\to 0}\nabla_t$ of 1-parameter families of regular connections for which the spectrum of the $p$-curvature has already been computed. In such cases, we may obtain the spectrum of the $p$-curvature of $\nabla$ from that of $\nabla_t$ by a limiting procedure, using that this spectrum depends on $t$ algebraically. Let us list some examples which can be handled in this way. 

\subsection{Irregular KZ connections} The KZ connection of \cite{EV}, Subsection 5.1 admits an irregular deformation considered in \cite{FMTV}. Namely, as in \cite{EV}, Section 5.1, let $\g$ be a finite dimensional simple Lie algebra over $\Bbb C$, $\Omega\in S^2\g$ the Casimir tensor, $\mathfrak h\subset \g$ a Cartan subalgebra, $h\in \mathfrak h$, and consider the connection 
\begin{equation}
\nabla(\hbar,h):=d-\sum_{i=1}^r \left(h^{(i)}+\hbar\left(\sum_{j\ne i}\frac{\Omega^{ij}}{x_i-x_j}\right)\right) dx_i
\end{equation} 
on the trivial bundle on $\Bbb C^r\setminus {\rm diagonals}$ with fiber being the weight space $(V_1\otimes...\otimes V_r)[\mu]$, where $V_i$ are finite dimensional representations of $\g$. This connection is obtained as a limiting case of the trigonometric KZ connection of \cite{EV}, Subsection 5.1.6 (for the usual quasitriangular structure $\bold r$ on $\g$) by replacing $x_i$ with $1+tx_i$, setting $s:=h/t$, and sending $t$ to $0$. 

We thus obtain  

\begin{proposition} For almost all $p$, the $p$-curvature operators $C_i(\hbar,h)(\bold x)$ of the reduction of $\nabla(\hbar,h)$ at $p$ are isospectral to
$$
C_i^*(\hbar,h)(\bold x):=(-h^p)^{(i)}+(\hbar-\hbar^p)\sum_{j\ne i}\frac{\Omega^{ij}}{x_i^p-x_j^p}.
$$ 
\end{proposition}  

\begin{remark}\label{bisp} Note that this is slightly different from the expression in Theorem \ref{main1}: in the first term instead of $h-h^p$ we have just $-h^p$, as in Theorem \ref{main3}. This is because we have set $s=h/t$ and sent $t$ to $0$. In fact, one can show that the limiting pencil $\nabla(\hbar,h)$ is mixed-periodic (and satisfies the assumption of Theorem \ref{main3}): it is periodic in $\hbar$ and infinitesimally periodic in $h$. Namely, the connection with respect to $h$ commuting with $\nabla(\hbar,h)$ is exactly the dynamical (or Casimir) connection of \cite{FMTV}, which is bispectrally dual to the irregular KZ connection. 
\end{remark} 

\subsection{Irregular Casimir connections} The Casimir connection with values 
$(V_1\otimes...\otimes V_n)[\mu]$ (\cite{EV}, 5.3.1) admits an irregular deformation, which was introduced in \cite{FMTV} and is (as just noted in Remark \ref{bisp}) bispectrally dual to the irregular KZ connection. To define it, we need to equip $V_1\otimes...\otimes V_n$ with the structure of a representation of the Lie algebra $\g[z]$, by considering the tensor product of evaluation representations $V_1(x_1)\otimes...\otimes V_n(x_n)$, 
where $\bold x=(x_1,...,x_n)\in \Bbb C^n$. Let $\mathfrak h_{\rm reg}$ be the set of regular elements of $\mathfrak h$, 
$\alpha_i,\omega_i^\vee$ be the simple roots and fundamental coweights of $\g$, 
$R_+$ be the set of positive roots of $\g$, $e_\alpha,e_{-\alpha}\in \g$
be the root elements corresponding to $\alpha\in R_+$. 
Then the irregular Casimir connection 
is the connection  on the trivial bundle over ${\mathfrak h}_{\rm reg}$ 
 with fiber $V:=(V_1(x_1)\otimes...\otimes V_n(x_n))[\mu]$ 
given by 
$$
\nabla(\hbar,\bold z) = d-\hbar\sum_{\alpha\in R_+}\frac{e_\alpha e_{-\alpha}+e_{-\alpha} e_\alpha}{2}\frac{d\alpha}{\alpha} +\sum_{i=1}^{{\rm rank}\g}(\omega_i^\vee\otimes z)d\alpha_i,
$$
where $\omega_i^\vee\otimes z\in \g[z]$ acts in $V$. 
This connection is a composition factor of the limit of the trigonometric 
Casimir connection (\cite{EV}, Subsection 5.3.2) 
under rescaling $h\mapsto th$ for $h\in \mathfrak h$, $\bold x\mapsto \hbar\bold x/t$, and sending $t$ to $0$, which corresponds to degeneration of the Yangian $Y(\g)$ to $U(\g[z])$ (it is a composition factor rather than the whole limit because the representations $V_j$ may not lift to $Y(\g)$). 

Let $C_i(\hbar,\bold x)$ be the $p$-curvature operator of the reduction of $\nabla(\hbar,\bold x)$ at $p$ corresponding to the vector $\omega_i^\vee\in \mathfrak h$. Then we get

\begin{proposition} For almost all $p$, the operators $C_i(\hbar,\bold x)(h)$ are isospectral to
$$
C_i^*(\hbar,\bold x)(h):=(\hbar-\hbar^p)\sum_{\alpha\in R_+}\frac{\alpha(\omega_i^\vee)}{\alpha(h)^p}\frac{e_\alpha e_{-\alpha}+e_{-\alpha} e_\alpha}{2} +\sum_{j=1}^nx_j^p(\omega_i^\vee)^{(j)}.
$$ 
\end{proposition}  

\begin{remark} Similarly to the irregular KZ connection, one can show that the irregular Casimir connection is mixed-periodic (and satisfies the assumption of Theorem \ref{main3}). Namely, it is periodic in $\hbar$ and infinitesimally periodic in $x_j$, with the commuting (bispectral dual) connection with respect to $\bold x$ being the irregular KZ connection. 
\end{remark} 

\subsection{Irregular Dunkl connections}\label{idc} 
Let $W$ be a finite Coxeter group 
with set of reflections $S$ and reflection representation $\mathfrak h$ of dimension $r$. Let $\bold c: S\to \Bbb C$ be a $W$-invariant function. The
Dunkl connection with $V=\Bbb CW$ (\cite{EV}, Subsection 5.2) admits an irregular deformation. Namely, for $\lambda\in \mathfrak h^*$ let $V_\lambda^0$ be the representation of $\Bbb CW\ltimes S\mathfrak h$ 
on $\Bbb CW$ where $W$ acts by left multiplication and 
$h\circ w=\lambda(w^{-1}h)w$, $w\in W$, $h\in \mathfrak h$. The irregular Dunkl connection is then 
the connection over the regular locus $\mathfrak h_{\rm reg}\subset \mathfrak h$ with fiber $V_\lambda^0$ given by
\begin{equation}\label{dunkl}
\nabla(\bold c,\lambda)=d-\sum_{i=1}^r u_i du_i^*-\sum_{w\in S}\bold c(w)\frac{d \alpha_w}{\alpha_w}w,
\end{equation} 
where $\alpha_w$ is a root corresponding to a reflection $w\in S$, $\lbrace u_i\rbrace$ is a basis of $\mathfrak h$, $u_i^*$ the dual basis of $\mathfrak h^*$, and $u_i$ acts in $V_\lambda^0$. If $W$ is a Weyl group then this is a limiting case of the trigonometric Dunkl connection $\nabla(\bold c)|_{V_{\lambda/t}}$ (\cite{EV}, Subsection 5.2.2) by  
zooming in with $h\mapsto th$, $h\in \mathfrak h$, and sending $t$ to $0$. 

We thus obtain

\begin{proposition}\label{dunpr} If $W$ is the Weyl group of a root system $R$ then 
for almost all $p$, the $p$-curvature operators $C_i(\bold c,\lambda)(h)$ of the reduction of $\nabla(\bold c,\lambda)$ at $p$ corresponding to the fundamental coweights $\omega_i^\vee\in \mathfrak h$ are isospectral to
$$
C_i^*(\bold c,\lambda)(h):=\left(-\omega_i^\vee+\sum_{w\in S}\bold (\bold c-\bold c^p)(w)\frac{\alpha_w(\omega_i^\vee)}{\alpha_w^p(h)}w\right)\biggl|_{V_{\lambda^p}^0}.
$$ 
\end{proposition}  

\begin{remark} As before, one can show that the irregular Dunkl connection is mixed-periodic (and satisfies the assumption of Theorem \ref{main3}). Namely, it is periodic in $\bold c$ and infinitesimally periodic in $\lambda$, and is bispectrally self-dual: 
the bispectrally dual connection is the irregular Dunkl connection with $\lambda$ and $h$ swapped. 
\end{remark} 

\begin{remark}\label{crg} The irregular Dunkl connection \eqref{dunkl} is a rare case of a nontrivial mixed-periodic pencil for which all the (infinitesimal) shift operators $A_j,a_k$ can be computed fairly explicitly. Namely, the irregular Dunkl connection  
can be described as the flat connection attached to the eigenvalue problem for the rational Calogero-Moser operators: 
$$
L_k(\bold c)\psi=\Lambda_k\psi,\ 1\le k\le r.
$$
In this realization the shift operators $A_j$ are the {\bf Dunkl-Opdam shift operators} $S_j(\bold c)$ (\cite{DO}, Subsection 3.5), which are explicit differential operators with rational coefficients which satisfy the identities 
$$
L_k(\bold c+\bold e_j)\circ S_j(\bold c)=S_j(\bold c)\circ L_k(\bold c)
$$
and hence map joint eigenfunctions of $L_k(\bold c)$ to joint eigenfunctions of $L_k(\bold c +\bold e_j)$, thus defining an isomorphism of the corresponding $D$-modules. 

Shift operators were generalized  to complex reflection groups by Berest and Chalykh (\cite{BC}), which can be used to prove an analog of Proposition \ref{dunpr} for all finite complex reflection groups (not just Weyl groups). One just needs to 
remember that reflection representations of complex reflection groups may be defined not over $\Bbb Q$ but over some number field $K$. 
This needs to be taken into account when computing Frobenius twists. 

The same holds for the trigonometric Dunkl connection (\cite{EV}, Subsection 5.2.2)
using its realization as the eigenvalue problem for trigonometric Calogero-Moser operators. Namely, the shift operators in this case are {\bf Opdam's shift operators} (\cite{O}).
\end{remark}

\subsection{Toda connections of finite type} \label{toda}
Toda connections are a special confluent limit of Dunkl-Cherednik connections  (see e.g. \cite{BMO}, Section 7). To explain this limit, for simplicity assume that the root system $R$ is irreducible and simply laced. 
Recall from \cite{EV}, Subsection 5.2.2 that in this case we have the degenerate affine Hecke algebra $\mathcal H_c$, $c\in \Bbb C$, generated by the Weyl group $W$ and $\mathfrak h$ 
with defining commutation relations
$$
s_i h-s_i(h)s_i=c\alpha_i(h),\ h\in \mathfrak h
$$
for simple reflections $s_i$ and simple roots $\alpha_i$, and for $\lambda\in \mathfrak h^*$ we define the induced representation $V_{c,\lambda}:=\mathcal H_c\otimes_{\rm Sym{\mathfrak h}}\Bbb C_\lambda$. Consider the torus $H=\mathfrak h/2\pi iP$ where $P$ is the (co)weight lattice. Recall that the Dunkl-Cherednik connection is the connection on the trivial bundle over the regular locus $H_{\rm reg}$ of $H$ with fiber $V_{c,\lambda}$ given by 
$$
\nabla(c,\lambda)=d-\sum_{i=1}^r \omega_i d\alpha_i-c\sum_{\alpha\in R_+}\frac{e^{\alpha}d\alpha}{1-e^{\alpha}}(s_\alpha-1), 
$$
where $\omega_i$ are the fundamental (co)weights.

In the Toda limit, the parameter $\lambda$ is kept fixed, while the parameter $c$ is sent to infinity. But in order for the limit to exist, we must also make a shift on $H$ and send it to infinity in a coordinated way with $c$, so the Weyl group symmetry is lost. To explain how exactly this works, introduce renormalized simple reflections $\overline s_i:=c^{-1}s_i$. 
They satisfy the defining relations 
$\overline s_i^2=c^{-2},\ \overline s_i h-s_i(h)\overline s_i=\alpha_i(h),\ h\in \mathfrak h $,
and the braid relations. Thus in the limit $c\to \infty$ these elements satisfy
the relations 
$$
\overline s_i^2=0,\ \overline s_i h- s_i(h)\overline s_i=\alpha_i(h), \ h\in \mathfrak h
$$
and the braid relations. These are defining relations of the 
{\bf nil-Hecke algebra} $\mathcal H_{\infty}$ of Kostant and Kumar. The representation $V_{c,\lambda}$ 
of $\mathcal H_c$ degenerates into the representation $V_{\infty,\lambda}:=\mathcal H_\infty\otimes_{\rm Sym{\mathfrak h}}\Bbb C_\lambda$ of $\mathcal H_\infty$.

Now we are ready to take the Toda limit. For this purpose, introduce
renormalized reflections $\overline s_\alpha=c^{-\ell(s_\alpha)}s_{\alpha}$, $\alpha\in R_+$, where $\ell$ denotes the length (these elements have nonzero limits in $\mathcal H_{\infty}$). In terms of these elements, we can write the Dunkl-Cherednik 
connection as
$$
\nabla(c,\lambda)=d-\sum_{i=1}^r \omega_i d\alpha_i-c\sum_{\alpha\in R_+}\frac{e^{\alpha}d\alpha}{1-e^{\alpha}}(c^{\ell(s_\alpha)}\overline{s_\alpha}-1), 
$$
We see that this has no limit as $c\to \infty$ unless we make a shift along $H$. 
So let us make the shift $g\mapsto gc^{-2\rho}$, $g\in H$. After this we have 
$$
\nabla(c,\lambda)=d-\sum_{i=1}^r \omega_i d\alpha_i-c\sum_{\alpha\in R_+}\frac{c^{-2|\alpha|} e^{\alpha}d\alpha}{1-c^{-2|\alpha|}e^{\alpha}}(c^{\ell(s_\alpha)}\overline{s_\alpha}-1),  
$$
where $|\alpha|$ is the height of $\alpha$. 
But it is known that $\ell(s_\alpha)=2|\alpha|-1$ 
(see \cite{BMO} and references therein), so we get 
$$
\nabla(c,\lambda)=d-\sum_{i=1}^r \omega_i d\alpha_i-\sum_{\alpha\in R_+}\frac{e^{\alpha}d\alpha}{1-c^{-2|\alpha|}e^{\alpha}}(\overline{s_\alpha}-c^{-\ell(s_\alpha)}).
$$ 
Now this has a finite limit as $c\to \infty$ given by 
$$
\nabla(\infty,\lambda)=d-\sum_{i=1}^r \omega_i d\alpha_i-\sum_{\alpha\in R_+}e^{\alpha}\overline{s_\alpha}d\alpha.
$$
This limit is the {\bf Toda connection} of the root system $R$. Its base 
is now the whole torus $H$ (the singularities disappeared in the limit), 
but it has irregular singularities at infinity.

Since $\nabla(c,\lambda)$ is periodic in $\lambda$, so is $\nabla(\infty,\lambda)$; 
in fact, as was shown by A. Givental, this connection arises as the equivariant quantum connection for the flag variety $G/B$ of the simple complex Lie group $G$ corresponding to $R$, so the shift operators for $\nabla(\infty,\lambda)$ can be constructed geometrically. Moreover, they can be computed explicitly either by taking the limit of Opdam shift operators, or from geometry. In fact, this was done in type $A$ in the more general case of partial flag varieties in \cite{TV}. 

However, in the terminology of \cite{EV}, $\nabla(\infty,\lambda)$ is not a pencil but rather an {\bf affine pencil}, in the sense that its dependence on $\lambda$ is linear, but inhomogeneous (owing to the last summand). This means that we cannot compute the spectrum of the $p$-curvature of $\nabla(\infty,\lambda)$ using Theorem \ref{main1}. 

But we can do this computation using the above limiting procedure and the knowledge of the spectrum of the $p$-curvature of the Dunkl-Cherednik connection. 
Namely, by Theorem \ref{main1}, the $p$-curvature operators $C_i(c,\lambda)$ corresponding to $\omega_i$ for the reduction of $\nabla(c,\lambda)$ at $p$
are isospectral to 
$$
C_i^*(c,\lambda)=-(\omega_i+(c^p-c)
\sum_{\alpha\in R_+}(\alpha,\omega_i)\tfrac{c^{-2p\rho}e^{p\alpha}}{1-c^{-2p\rho}e^{p\alpha}}(s_\alpha-1))\big|_{V_{c^p-c,\lambda^p-\lambda}}.
$$
So rewriting this in terms of $\overline s_i=\frac{s_i}{c^p-c}$ (as $c$ 
has been replaced by $c^p-c$) and taking the limit $c\to \infty$, we obtain the following proposition. 

\begin{proposition} The $p$-curvature operators $C_i(\infty,\lambda)(g)$ corresponding to $\omega_i$ for the reduction of $\nabla(\infty,\lambda)$ at $p$
are isospectral to 
$$
C_i^*(\infty,\lambda)(g)=-(\omega_i+
\sum_{\alpha\in R_+}(\alpha,\omega_i)\alpha(g)^p\overline s_\alpha)\big|_{V_{\infty,\lambda^p-\lambda}},\ g\in H. 
$$
\end{proposition} 

\begin{remark} There is a geometric explanation for this limiting process given in \cite{BMO}, Section 8: the parameter is given by the weight of a conical action on $T^*(G/B)$, and taking the above limit amounts to changing the deformation-obstruction theory on $T^*(G/B)$ used to define the (equivariant) Gromov-Witten invariants to that on the zero section $G/B$. 
  \end{remark}
  
\subsection{Spectrum of $p$-curvature of Dunkl connections and rational Cherednik algebras}\label{rca} Proposition \ref{dunpr} allows one to describe the eigenvalues of the $p$-curvature of $\nabla(\bold c,\lambda)$ a bit more explicitly, using the representation theory of classical rational Cherednik algebras.\footnote{For basics on rational Cherednik algebras, we refer the reader to \cite{EG,Et2,BR}.} Indeed, note that 
$$
C_i^*(\bold c,\lambda)(h)=-G_i(\bold c^p-\bold c,h^p,\lambda^p),
$$ 
where $G_i(\bold c,h,\lambda)$ are the {\bf Gaudin operators} for $W$ (\cite{BR}, Subsection 5.5.B), whose eigenvalues can be described using the techniques of \cite{BR}. 

Namely, let $H_{t,\bold c}$ be the {\bf rational Cherednik algebra} attached to $\mathfrak h,W$ with triangular decomposition $H_{t,\bold c}=S\mathfrak h^*\otimes \bold kW\otimes S\mathfrak h$ (\cite{EG}; \cite{BR}, Section 3). Let 
$$
M(\lambda):=H_{t,\bold c}/H_{t,\bold c}(\mathfrak h-\lambda(\mathfrak h))\cong S\mathfrak h^*\otimes \bold kW
$$ 
be the Whittaker module over $H_{t,\bold c}$ with highest weight $\lambda$. 
So in the {\bf classical case} $t=0$ (\cite{BR}, Section 4), the central subalgebra $(S\mathfrak h)^W\subset  H_{0,\bold c}$ acts on $M(\lambda)$ by the character $\lambda$. Given an element $h\in  \mathfrak h$, consider the space $M(h,\lambda):=M(\lambda)/(\mathfrak h^*-h(\mathfrak h^*))M(\lambda)\cong \bold kW$. Let $\delta\in S\mathfrak h^*$ be the discriminant of $W$ and 
recall the {\bf classical Dunkl operator representation} 
$\Theta: H_{0,\bold c}[\delta^{-1}]\cong \bold kW\ltimes \bold k[\mathfrak h_{\rm reg}\times\mathfrak h^*]$. For $h\in \mathfrak h_{\rm reg}$, 
the space $\Theta^{-1}(\mathfrak h)$ acts on $M(h,\lambda)$ 
by the Gaudin operators: $\omega_i^\vee\mapsto G_i(\bold c,h,\lambda)$ (\cite{BR}, Subsection 5.5). 

To compute the joint eigenvalues $\mu_j\in \mathfrak h^*$ of the Gaudin operators $(1\le j\le |W|)$, let $\overline h:=Wh$, and note that the space $M(\overline h,\lambda):=\oplus_{h\in \overline h}M(h,\lambda)$ is a left $H_{0,\bold c}$-module; indeed, it is the quotient of $M(\lambda)$ by the character $\overline h$ of the central subalgebra $(S\mathfrak{h}^*)^W=\bold k[\mathfrak h/W]$. Also let $P_k$ be independent homogeneous generators of the algebra $\bold k[\mathfrak h/W]$, and let $H_k\in \bold k[\mathfrak h_{\rm reg}\times\mathfrak h^*]$ be {\bf the classical Calogero-Moser hamiltonian} corresponding to $P_k$. Then by Lemma 2.2 of \cite{EFMV}, 
$\Theta(P_k)=H_k$, hence $P_k=\Theta^{-1}(H_k)$. 
Thus, comparing the actions of both sides of this equality on $M(\overline h,\lambda)$, we get 

\begin{proposition}\label{gau} The spectrum of the Gaudin operators on $M(h,\lambda)$ is the multiset of solutions $\mu$ of the equations
$$
P_k(\lambda)=H_k(\bold c,h,\mu), 1\le k\le r.
$$
Namely, for generic $h,\lambda$, the eigenvalues are distinct, while for special values of $h,\lambda$, some of the solutions coalesce and come with a multiplicity.
\end{proposition} 

Proposition \ref{gau} may be seen as a comparison of two constructions
of generic irreducible representations of $H_{0,\bold c}$. Namely, recall that $H_{t,\bold c}$ contains the Euler element $\bold h$ such that $[\bold h,W]=0$, $[\bold h,x]=tx,[\bold h,y]=-ty$ for $x\in \mathfrak{h}^*,y\in \mathfrak{h}$ (\cite{EG}; \cite{BR}, Subsection 3.3); so for $t=0$, $\bold h$ is central. By Subsection 10.1.D of \cite{BR}, for generic $h,\lambda$ the element $\bold h$ has $|W|$ 
distinct eigenvalues on $M(\overline h,\lambda)$, and 
the eigenspace $L(\overline h,\lambda,\beta)$ with each eigenvalue $\beta$ is an irreducible $H_{0,c}$-module of dimension $|W|$. Also, right multiplication by $w\in W$ defines an isomorphism $L(\overline h,\lambda,\beta)\cong L(\overline h,w\lambda,\beta)$, so $L(\overline h,\lambda,\beta)$ depends only on the $W$-orbit $\overline \lambda=W\lambda$. Moreover, this gives a bijective parametrization of generic irreducible $H_{0,c}$-modules by triples $(\overline h,\overline \lambda,\beta)$ (where $\beta$ takes $|W|$ values for generic $\overline h,\overline \lambda$). These statements all follow from the fact that $H_{0,\bold c}$ is finite over its center $Z_{0,\bold c}$ (a domain and a free module over its subalgebra $(S\mathfrak h^*)^W\otimes (S\mathfrak h)^W$ of rank $|W|$), and is generically an Azumaya algebra of degree $|W|$ over $Z_{0,\bold c}$ (\cite{EG}; \cite{BR}, Section 4).

On the other hand, using the isomorphism $\Theta$, we may 
parametrize generic irreducible $H_{0,c}$-modules by pairs 
$(h,\mu)\in \mathfrak h_{\rm reg}\times \mathfrak h^*$ modulo 
the diagonal action of $W$. We denote the corresponding irreducible induced module by $N(h,\mu)$. It is easy to see that $\bold h$ acts on $N(h,\mu)$
by the eigenvalue $\mu(h)$. 

Now we may ask when 
$N(h,\mu)\cong L(\overline h,\lambda,\beta)$. 
The answer to this question is provided by the following direct corollary of 
Proposition \ref{gau}. 

 \begin{corollary}\label{gau1} One has $N(h,\mu)\cong L(\overline h,\lambda,\beta)$ if and only if $\mu$ is a joint eigenvalue 
 of the Gaudin operators $G_i(\bold c,h,\lambda)$, i.e., a solution of the equations
$$
P_k(\lambda)=H_k(\bold c,h,\mu),\ 1\le k\le r,
$$
such that $\beta=\mu(h)$. 
\end{corollary} 

Finally, applying Proposition \ref{gau} to $p$-curvature, we obtain 

\begin{proposition} The spectrum of the $p$-curvature of $\nabla(\bold c,\lambda)(h)$ is the multiset of solutions $\mu$ of the equation 
$$
P_k(\lambda^p)=H_k(\bold c^p-\bold c,h^p,-\mu), 1\le k\le r,
$$
where each solution is counted with its multiplicity. 
\end{proposition}

In view of Remark \ref{crg}, these results extend mutatis mutandis to the case when $W$ is a complex reflection group. Also, they extend staightforwardly to the trigonometric case, for the Dunkl-Cherednik connections (\cite{EV}, Subsection 5.2.2); in this case $H_k$ are the trigonometric classical Calogero-Moser hamiltonians.\footnote{More precisely, for this we need to extend Lemma 2.2 of \cite{EFMV} to the trigonometric case, i.e., show that $\Theta(P_k)$ is a function (has no 
nontrivial elements of $W$). To do so, note that $P_k$ 
is central in the degenerate affine Hecke algebra generated by 
$\mathfrak h$ and $W$, hence $\Theta(P_k)$ commutes with 
$W$ and functions on $\mathfrak h^*$. But this implies that 
$\Theta(P_k)$ contains no nontrivial elements of $W$.} Finally, by taking a limit from the trigonometric case, they 
extend to the case of Toda connections described in Subsection \ref{toda}, 
with $H_k$ being the classical Toda hamiltonians. In this case, the role of 
the Cherednik algebra is played by the degenerate nil-DAHA (\cite{G}).

\begin{example} Let $W=S_n$. In this case $\mathfrak h\cong \mathfrak h^*$ is the space of vectors $(x_1,...,x_n)$ such that $\sum_i x_i=0$ and $\bold c$ is a single parameter $c$. Recall that Moser's {\bf Lax  matrix} of the classical rational Calogero-Moser system 
is 
$$
\Bbb L(h,\mu)=\sum_{j=1}^n \mu_jE_{jj}+\sum_{i\ne j}\frac{c}{x_i-x_j}E_{ij},
$$
where $h=(x_1,...,x_n)$, $\mu=(\mu_1,...,\mu_n)$, and $E_{ij}$ are elementary matrices (see e.g. \cite{Et2}, Proposition 2.6). 
The hamiltonian $H_k$ of the Calogero-Moser system 
attached to the elementary symmetric function $e_k$ 
is then $H_k(h,\mu)={\rm Tr}\wedge^k \Bbb L(h,\mu)$ (\cite{Et2}, Subsection 2.7), so the equations 
of Corollary \ref{gau1} take the form 
$$
e_k(\lambda)={\rm Tr}\wedge^k \Bbb L(h,\mu),\ 2\le k\le n.
$$
For example, for $n=2$ we have $h=\frac{1}{2}(x,-x)$, $\mu=(p,-p)$, 
$\lambda=(y,-y)$, so we get the equation 
$$
y^2=p^2-\frac{c^2}{x^2}, 
$$
thus the eigenvalues of the Gaudin operator $G$ are $p=\pm \sqrt{y^2+\frac{c^2}{x^2}}$, which is also easy to check directly, since 
$G=\begin{pmatrix} y & \frac{c}{x}\\ \frac{c}{x} & -y\end{pmatrix}$. 

For $n=3$, we have $\sum_j x_j=\sum_j \lambda_j=\sum_j \mu_j=0$ 
and $\mu_j$ are determined for given $x_j,\lambda_j$ from the equations
$$
\lambda_1\lambda_2+\lambda_1\lambda_3+\lambda_2\lambda_3=\mu_1\mu_2+\mu_1\mu_3+\mu_2\mu_3
+\frac{c^2}{(x_2-x_3)^2}+\frac{c^2}{(x_3-x_1)^2}+\frac{c^2}{(x_1-x_2)^2},
$$
$$
\lambda_1\lambda_2\lambda_3=\mu_1\mu_2\mu_3+\frac{c^2\mu_1}{(x_2-x_3)^2}+\frac{c^2\mu_2}{(x_3-x_1)^2}+\frac{c^2\mu_3}{(x_1-x_2)^2}.
$$
This system reduces to a degree $6$ polynomial equation in one variable over the field $\bold k(\lambda,x)$ whose Galois group is $S_6$ (so the degree $6$ field extension defined by this polynomial is very far from being Galois). 

The case when $W$ is a general dihedral group $I_m$ is discussed in detail in \cite{B}. In this case, one gets a system of two equations for $\mu$, one of degree $2$ and another of degree $m$, which reduces to an equation in  one variable of degree $2m$, with maximal possible Galois group $S_{2m}$ if $m$ is odd (\cite{B}, Theorem 6.1). The above example of type $A_2$ corresponds to $m=3$. 
\end{example}

Let us now describe the relation of the spectrum of $p$-curvature of irregular Dunkl connections to representation theory of quantum rational Cherednik algebras $H_{1,\bold c}$ in characteristic $p$, which arises from comparison of two constructions of generic irreducible representations of $H_{1,\bold c}$. 
Similarly to the case of $H_{0,\bold c}$, define the left $H_{1,\bold c}$-module 
$M(\overline h,\lambda)$ to be the quotient of the Whittaker module $M(\lambda)$ by the character $\overline h$ of the central subalgebra 
$((S\mathfrak h)^W)^p\subset H_{1,\bold c}$. We also have the central element 
$\bold h^{(1)}:=\bold h^p-\bold h\in H_{1,\bold c}$, and it is shown similarly to 
\cite{BR}, Subsection 10.1.D that it has $|W|$ distinct eigenvalues on 
 $M(\overline h,\lambda)$, and for each eigenvalue $\beta$ 
 the corresponding eigenspace $L(\overline h,\lambda,\beta)$ is an 
 irreducible $H_{1,\bold c}$-module, now of dimension $p^r|W|$.
 As before, $L(\overline h,\lambda,\beta)$ depends only on $\overline \lambda=W\lambda$, and this provides a bijective parametrization of generic irreducible modules by triples $(\overline h,\overline \lambda,\beta)$. All this follows from the facts that $H_{1,\bold c}$ is finite over its center $Z_{1,\bold c}$ (a domain and a free module over its subalgebra $((S\mathfrak h^*)^W)^p\otimes ((S\mathfrak h)^W)^p$ of rank $|W|$), and is generically an Azumaya algebra of degree $p^r|W|$ over $Z_{1,\bold c}$ (see e.g. \cite{Et,BrC}). 
 
On the other hand, recall the {\bf quantum Dunkl operator representation} 
$$
\Theta: H_{1,\bold c}[\delta^{-1}]\cong \bold kW\ltimes D(\mathfrak h_{\rm reg}),$$ 
where $D(\mathfrak h_{\rm reg})$ is the algebra of differential operators on 
$\mathfrak h_{\rm reg}$. For $\mu\in \mathfrak{h}^*$ we have the $D(\mathfrak{h}_{\rm reg})$-module 
$$
\bold k_\mu[\mathfrak{h}_{\rm reg}]:=
D(\mathfrak{h}_{\rm reg})/D(\mathfrak{h}_{\rm reg})(\partial_y-\mu(y),y\in \mathfrak{h}).
$$
Thus we can consider the induced $W\ltimes D(\mathfrak{h}_{\rm reg})$-module $\bold kW\otimes \bold k_\mu[\mathfrak{h}_{\rm reg}]$ and the restriction to $H_{1,\bold c}$ of its pullback via $\Theta$, which we denote by $N(\mu)$. Note that $N(\mu)$ is a free $\bold k[\mathfrak{h}_{\rm reg}/W]^p$-module of rank $p^r|W|^2$, and the action of $H_{1,\bold c}$ is linear over this algebra. Thus for every $\overline h\in \mathfrak{h}_{\rm reg}/W$, we have the fiber $N(\overline h,\mu)$ of $N(\mu)$ at $\overline h$, which is a representation of $H_{1,\bold c}$ of dimension $|W|^2p^r$. 
Moreover, it is easy to see that $\bold h^{(1)}$ acts on $N(\overline h,\mu)$ with eigenvalues $\mu(h)^p$, $h\in \overline h$, so we can define $N(h,\mu)$ 
as the $\mu(h)^p$-eigenspace of $\bold h^{(1)}$ in $N(\overline h,\mu)$. Then generically $N(h,\mu)$ is irreducible, $N(h,\mu)\cong N(wh,w\mu)$ 
for $w\in W$, and every sufficiently generic irreducible representation 
of $H_{1,\bold c}$ is isomorphic to $N(h,\mu)$ for a unique pair $(h,\mu)$ 
up to diagonal action of $W$. 

So, as before, we may ask when $N(h,\mu)\cong L(\overline h,\lambda,\beta)$ 
The answer to this question is provided by the following proposition, which follows from Proposition \ref{dunpr}.

\begin{proposition} We have $N(h,\mu)=L(\overline h,\lambda,\beta)$ 
if and only if $-\mu^p$ is an eigenvalue of the 
$p$-curvature $C(\bold c,\lambda)(h)$ 
of the connection $\nabla(\bold c,\lambda)$ 
at $h$, i.e., 
$$
P_k(\lambda^p)=H_k(\bold c^p-\bold c,h^p,\mu^p), 1\le k\le r,
$$
and $\beta=\mu(h)^p$. 
\end{proposition} 

Moreover, similarly to Subsection \ref{rca}, this result extends to the case of complex reflection groups, the trigonometric case, and the Toda case. 
  
 \section{Going beyond pencils}
 
Unfortunately, the pencil condition (that the connection form is linear homogeneous in the parameters) used in the proof of Theorem \ref{main1} is quite restrictive, even though it allows many interesting examples. This motivated us in \cite{EV}, Subsection 4.8 to generalize the theory 
of periodic pencils to {\it pseudo-pencils}, which is a much wider class of families of flat connections. Namely, recall (\cite{EV}, Definition 4.29) that a 1-parameter polynomial family of flat connections $\nabla(s)=d-B(s)$ on $X$ is a {\bf pseudo-pencil} if it is gauge equivalent 
over $\bold k(s)(X)=\bold k(X)(s)$ (or, equivalently, over $\bold k(X)((s))$) to a family of the form $d-sB_\infty(s)$, where $B_\infty(s)$ is regular at $s=\infty$. It is shown in \cite{EV}, Subsection 4.8 that periodic pseudo-pencils share many of the remarkable properties of periodic pencils, while including many more examples, e.g. ones coming from Hodge theory. In this section we generalize Theorem \ref{main1} to pseudo-pencils (satisfying some additional conditions).\footnote{In \cite{EV}, Subsection 4.8 we actually considered the more general case of multiparameter pseudo-pencils, but here for simplicity we will restrict ourselves to 1-parameter ones.}  

\subsection{Trivializability at $s=0$} 
Let $X$ be a smooth variety over a field $\bold k$ (for simplicity, 1-dimensional, i.e., a curve, but this assumption is not essential). Let $\nabla(s)$ be a polynomial family of (flat) connections on $X$.

\begin{definition} $\nabla(s)$ is {\it trivializable at $0$} if there exists a gauge transformation over $\bold k(X)(s)$ (or, equivalently, $\bold k(X)((s))$, i.e., allowing poles at $s=0$) which transforms $\nabla(s)$ to the form $\nabla_*(s)=d-sB_0(s)$, where $B_0(s)$ is regular at $s=0$ (i.e., $\nabla_*(0)$ is trivial). 
\end{definition}

For example, if $\nabla(s)$ is a pencil, it is automatically trivializable (in fact, trivial) at $0$. 

\begin{proposition} A connection $\nabla_0$ on $X$ admits a  
trivializable at $0$ deformation $\nabla(s)$ if and only if 
$\nabla_0$ is unipotent (i.e., admits an invariant filtration with associated graded the trivial connection).
\end{proposition} 

\begin{proof} If $\nabla_0$ is unipotent then 
$\nabla_0=d-B$, where $B$ 
is strictly upper triangular in some basis. 
Then conjugation of $\nabla_0$ by ${\rm diag}(1,s,...,s^{N-1})$ (in this basis), brings it to the form $d-sB_0(s)$ for $B_0(s)$ regular at $s=0$. It follows that $\nabla(s):=\nabla_0$ (constant family) is trivializable at $0$. 

Conversely, if $\nabla(s)$ is trivializable at $0$ 
then $\nabla_0=\nabla(0)$ is unipotent, since the composition series 
of $\nabla(0)$ does not change under gauge transformations 
of $\nabla(s)$ over $\bold k(X)((s))$ (even though its isomorphism 
class might change). 
\end{proof} 

Assume now that $\bold k$ is of characteristic $p$. 

\begin{proposition}\label{p2} The following conditions are equivalent: 

(i) $\nabla(s)$ is trivializable at $0$. 

(ii) The $p$-curvature $C(s)$ of $\nabla(s)$ is conjugate over $\bold k(X)((s))$ to 
$C_*(s)\in s\cdot {\rm Mat}_N(\bold k(X)[[s]])$. 

(iii) The characteristic polynomial of $C(s)/s$ has coefficients in $\bold k(X)[[s]]$. 
\end{proposition} 
 
\begin{proof} We will use the following lemma. 
 
\begin{lemma}\label{l1} Let  $K$ be a field and $A\in {\rm Mat}_N(K((s)))$ 
be such that the coefficients of the characteristic polynomial $\chi_A$ of $A$ are in $K[[s]]$. Then $A$ is conjugate by an element of $GL_N(K((s)))$ to an element of ${\rm Mat}_N(K[[s]])$. 
\end{lemma} 

\begin{proof}
By the Cayley-Hamilton theorem we have $\chi_A(A)=0$, so $A$ is integral over $K[[s]]$. Thus there is a $K[[s]]$-lattice $L$
in  $K((s))^N$ which is $A$-stable. Writing $A$ in a basis of this lattice, we obtain a matrix conjugate to $A$ which has entries in $K[[s]]$. 
\end{proof} 

Now we are ready to prove the proposition. The $p$-curvature of $\nabla(s)$ is conjugate to the $p$-curvature of $\nabla_*(s)$, so 
(i) implies (iii). The equivalence of (iii) and (ii) follows from Lemma \ref{l1}. So it remains to show that (iii) 
implies (i). Consider the standard $\bold k(X)[[s]]$-lattice $L=\bold k(X)[[s]]^N$ for $\nabla(s)$, and let $M$ be the submodule generated inside $\bold k(X)((s))^N$ by $(C(s)/s)^jL$, $j\ge 0$. By assumption and the Cayley-Hamilton theorem, $M$ is a finitely generated $\bold k(X)[[s]]$-module. 
Thus it is a lattice in $\bold k(X)((s))^N$. Moreover, $M$ is invariant under $\nabla(s)$ since $\nabla(s)$ commutes with $C(s)$. Choosing a basis of $M$, we obtain a gauge in which $\nabla(s)$ and $C(s)/s$ are regular. 
But then $C(0)=0$. So, in a suitable gauge (regular at $s=0$), we have $\nabla(0)=d$, hence $\nabla(s)=d-sB_0(s)$, proving (i). 
\end{proof} 

\begin{remark} $\nabla(s)$ may fail to be trivializable at $0$ even if $\nabla_0$ is unipotent.
 For example, the connection 
$$
\nabla(s)=d-\frac{1}{x}\begin{pmatrix} 0& 1\\ s& 0 \end{pmatrix}
$$
is not trivializable at $0$. Indeed, for $p$ odd the $p$-curvature of $\nabla(s)$ is
$$
C(s)=\frac{1}{x^p}\begin{pmatrix} 0& 1-s^{\frac{p-1}{2}}\\ s-s^{\frac{p+1}{2}}& 0 \end{pmatrix}, 
$$ 
so ${\rm Tr}\wedge^2 C(s)=-\frac{s}{x^{2p}}+O(s^2)$, and 
for $p=2$ it is $C(s)=\frac{1}{x^2}\begin{pmatrix} s& 1\\ s& s \end{pmatrix}$, 
so ${\rm Tr}\wedge^2 C(s)=\frac{s+s^2}{x^4}$, violating condition (iii) in  Proposition \ref{p2}. This also implies that $\nabla(s)$ 
is not trivializable at $0$ in characteristic $0$. 
\end{remark} 

\begin{definition} Let $\nabla_0$ be a connection on $X$ over a field $\bold k$ of characteristic zero, and let
 $\nabla(s)$ be a polynomial deformation of $\nabla_0$. We call this deformation {\bf globally trivializable at $0$} if it is trivializable at $0$ upon reduction modulo almost all primes. 
\end{definition} 

\begin{remark} Note that if $\nabla(s)$ is globally trivializable at $0$ then 
$\nabla_0$ need not be unipotent (in fact, it can be irreducible of arbitrary rank). Indeed, if $\nabla(s)$ is a periodic pencil then $\nabla(s+b)$ is globally trivializable at $0$ for any $b\in \Bbb Q$. 
Namely, for large enough $p$, $b$ is represented by an integer $0\le b_p\le p-1$, and the gauge transformation trivializing $\nabla(s+b)=\nabla(s+b_p)$ at $0$ in characteristic $p$ is the product of $b_p$ shift operators which conjugates $\nabla(s+b_p)$ to $\nabla(s)$. In fact, this argument shows that the statement holds more generally, for 
any periodic polynomial family $\nabla(s)$ trivializable at $0$. 

In particular, we see that trivializability at $0$ implies global trivializability at $0$ but not vice versa.  
\end{remark} 

\subsection{Periodic pseudo-pencils} Now let $\bold k$ be of characterisic $p$. 
Let $\nabla(s)$ be a periodic 1-parameter family of (flat) connections on $X$. 

The following theorem is a generalization of Theorem \ref{main1}.

\begin{theorem} If $\nabla(s)$ is a pseudo-pencil trivializable at $0$ then its $p$-curvature is isospectral to $(s-s^p)B_\infty^{(1)}$ and also to $-(s-s^p)\partial^{p-1}B_0$. In particular, $B_\infty^{(1)}$ is isospectral to $-\partial^{p-1}B_0$.\footnote{For example, if $\partial-sB$ is a periodic pencil then 
$B^{(1)}$ is isospectral to $-\partial^{p-1}B$, a nontrivial condition on the matrix-valued 1-form $B$. E.g., if $B$ has generically simple spectrum then 
this condition says that there exists $g\in GL(V)(\bold k(X))$ such that 
$-\partial^{p-1}B=(gBg^{-1})^{(1)}$. By the Lang-Steinberg theorem, there exists 
$h\in GL(V)(\overline{\bold k(X)})$ such that $g=h^{-1}h^{(1)}$. Then 
$\overline B:=h^{(1)}B(h^{(1)})^{-1}$ satisfies the equation $\overline B^{(1)}=-\partial^{p-1}\overline B$, which is equivalent to saying that $B$ has simple 
poles and residues having entries in $\Bbb F_p$.}
\end{theorem} 

\begin{proof} Let $C(s)$ be the $p$-curvature of $\nabla(s)$. This is a polynomial of $s$, 
so 
$$
c_i(s):={\rm Tr}\wedge^i C(s),\ c_i\in \bold k[s].
$$ 
By the pseudo-pencil condition, $\deg c_i\le pi$. 
By periodicity, $c_i(s)=\gamma_i(s-s^p)$, $\gamma_i\in \bold k[s]$, where $\deg(\gamma_i)\le i$. By 
trivializability at $0$, $\gamma_i(u)$ is divisible by $u^i$. Thus $c_i(s)=c_i^0 (s-s^p)^i$, $c_i^0\in \bold k$. So by sending $s$ to infinity we find that $C(s)$ is isospectral to $(s-s^p)B_\infty^{(1)}$ and by sending $s$ to $0$ we find that $C(s)$ is isospectral to $-(s-s^p)\partial^{p-1}B_0$. 
\end{proof} 

\begin{corollary}\label{glotri} If $\nabla(s)=d-B(s)$ is a pseudo-pencil over a field of characteristic $0$ that is globally trivializable at $0$ then for almost all $p$ the $p$-curvature of its reduction to characteristic $p$ is isospectral to $(s-s^p)B_\infty^{(1)}$ and also to $-(s-s^p)\partial^{p-1}B_0$.
\end{corollary} 

\subsection{A question} 

Let $\pi: Y\to X$ be a smooth morphism 
between smooth complex algebraic varieties, and 
let $f: Y\to \Bbb A^1\setminus 0$ be a regular function (say, all defined over $\overline{\Bbb Q}$). For $s\in \Bbb C$, 
consider the generalized Gauss-Manin connection $\nabla(s)$ on $H^i(\pi_*f^s)$ 
(modulo $s$-torsion) for some $i$, where $f^s$ denotes the de Rham local system generated by the multivalued function $f^s$ (i.e., a one-parameter motivic family). As noted in Subsection \ref{motfam}, $\nabla(s)$ is a periodic family: $\nabla(s+1)$ is conjugate to $\nabla(s)$. 

Corollary \ref{glotri} motivates the following question. 

\begin{question} \label{que}
1. Is $\nabla(s)$ a pseudo-pencil? 

2. Is $\nabla(s)$ globally trivializable at $0$? (equivalently, is this so for  
$\nabla(s+r)$, $r\in \Bbb Q$?) 
\end{question} 

Note that a positive answer to Question \ref{que}(2) would be a generalization of 
Katz's theorem (Theorem \ref{katzth}) to the case of generalized Gauss-Manin connections. 

If the general answer is ``no", it would be interesting to find nice sufficient conditions for the answer to be ``yes", so that one can apply Corollary \ref{glotri}.

\section{Periodic families of additive and multiplicative difference connections} 

\subsection{Difference connections} 
In this section we would like to generalize the theory of periodic families of connections 
to difference connections (additive and multiplicative). The variety $X$ for these generalizations 
will be the affine space $\bold k^r$ over $\bold k$ with coordinates $x_1,...,x_r$, and we will allow connection forms with poles. 

Let $V$ be a finite dimensional $\bold k$-vector space. Define the additive translation operators $T_i$ by
$$
(T_if)(x_1,...,x_i,...,x_r):=f(x_1,...,x_i+1,...,x_r),
$$
or $(T_if)(\bold x)=f(\bold x+\bold e_i)$, 
and for $\mathbbm{q}\in \bold k^\times$ the multiplicative translation operators
$\bold T_{i,\mathbbm{q}}=\bold T_i$ by
$$
(\bold T_if)(x_1,...,x_i,...,x_r):=f(x_1,...,\mathbbm{q} x_i,...,x_r)
$$
or $(\bold T_if)(\bold x)=f(\mathbbm{q}^{\bold e_i}\bold x)$ (where $\mathbbm{q}^{(m_1,...,m_r)}:=(\mathbbm{q}^{m_1},...,\mathbbm{q}^{m_r})$), 
acting on $\bold k(X)$. 

\begin{definition} An {\bf additive difference connection} $\nabla$ on $X$ with values in $V$ 
is a collection of difference operators  
$$
\nabla_i=T_i^{-1}B_i,
$$
where $B_i: \bold k^r\to GL(V)$ are rational functions. 
Similarly, a {\bf multiplicative $\mathbbm{q}$-difference connection} $\nabla$ on $X$ with values in $V$ 
is a collection of multiplicative $\mathbbm{q}$-difference operators  
$$
\nabla_i=\bold T_i^{-1}B_i.
$$
Such a connection $\nabla$ is said to be {\bf flat} if 
$[\nabla_i,\nabla_j]=0$. If $\bold k=\Bbb C$, a
{\bf flat meromorphic section} of $\nabla$ 
over an open subset $U\subset \Bbb C^r$ invariant under 
$x_i\mapsto x_i+1$, respectively $x_i\mapsto \mathbbm{q} x_i$ is a meromorphic function 
$f: U\to V$ such that $\nabla_i f=f$, i.e., $T_if=B_if$, respectively 
$\bold T_if=B_if$. 
\end{definition} 

\subsection{$p$-curvature of difference connections} 
The notion of {\bf $p$-curvature} makes sense for flat additive difference connections if 
${\rm char}(\bold k)=p$, and for flat multiplicative $\mathbbm{q}$-difference connections 
if $\mathbbm{q}$ is a primitive $p$-th root of $1$, where $p\ge 2$ is an integer (so in any characteristic not dividing $p$). Namely, in these cases we make

\begin{definition} The $p$-curvature of $\nabla$ is the collection 
of commuting operators $C_i:=\nabla_i^p$. 
\end{definition} 

These operators are just $\bold k(X)$-linear automorphisms 
of $\bold k(X)\otimes V$: namely, in the additive case 
$$
C_i(\bold x)=\prod_{j=p-1}^{0} B_i(\bold x+j\bold e_i)=B_i(\bold x+(p-1)\bold e_j)...B_i(\bold x+\bold e_i)B_i(\bold x), 
$$
and in the multiplicative case 
$$
C_i(\bold x)=\prod_{j=p-1}^{0} B_i(\mathbbm{q}^{j\bold e_i}\bold x)=B_i(\mathbbm{q}^{(p-1)\bold e_i}\bold x)...B_i(\mathbbm{q}^{\bold e_i}\bold x)B_i(\bold x).
$$

When a flat additive difference connection $\nabla$ degenerates to an ordinary connection $\nabla^0$, the $p$-curvature of $\nabla$ degenerates to the $p$-curvature of $\nabla^0$. For example, 
consider a one-dimensional additive difference connection $\nabla=T^{-1}B$. We may rescale the 
coordinate $x$ by $t$, then we obtain the operator 
$$
(\nabla_t f)(x)=B(\tfrac{x-t}{t})f(x-t).
$$
Assume that $B(\frac{x-t}{t})=1+tb(x)+O(t^2)$, then we have 
$$
(1-\nabla_t)f=t(\partial-b)f+O(t^2), 
$$
so if $C_t=\nabla_t^p$ then 
$$
(1-C_t)f=(1-\nabla_t)^pf=t^p(\partial-b)^pf+O(t^{p+1}).
$$
Thus if $C^0$ is the $p$-curvature of the ordinary connection  
$\nabla^0:=\partial-b=\lim_{t\to 0}\frac{1-\nabla_t}{t}$ then 
$$
C^0=\lim_{t\to 0}\frac{1-C_t}{t^p}.
$$
Similarly, degeneration of a multiplicative $\mathbbm{q}$-difference connection in characteristic zero to a usual one (accompanied by reduction to characteristic $p$) preserves the $p$-curvature. Even more obviously, degeneration of a multiplicative $\mathbbm{q}$-difference connection in characteristic zero to an additive one (accompanied by reduction to characteristic $p$) also preserves the $p$-curvature. 

\subsection{Periodic families of difference connections} 

We now want to generalize the notion of a periodic family of flat connections to the difference case. Motivated by applications, we will allow both additive and multiplicative parameters (but additive parameters will only appear for additive difference connections). 

\begin{definition} 
An {\bf polynomial family} of flat additive (respectively, multiplicative) difference connections on $X$ with values in $V$ is a family of difference connections $\nabla(\bold s,\bold t)=(\nabla_i(\bold s,\bold t))$ such that $\nabla_i^{\pm 1}$ are polynomials in $\bold s\in \bold k^n$ and Laurent polynomials in $\bold t\in (\bold k^\times)^m$. 
\end{definition}

\begin{definition} A family $\nabla(\bold s,\bold t)$ of additive difference connections is {\bf (mixed-)periodic} if there exist operators 
$A_j\in GL(V)(\bold k(\bold s,\bold t)(X))$, $1\le j\le n$ such that 
$$
\nabla(\bold s+\bold e_j,\bold t)\circ A_j(\bold s,\bold t)=A_j(\bold s,\bold t)\circ \nabla(\bold s,\bold t),\ 1\le j\le n
$$
and $a_j\in {\rm End}V\otimes \kk(\bold s,\bold t)(X),\ n+1\le j\le n+m$
such that
$$
[t_j\partial_{t_j}-a_j(\bold s,\bold t),\nabla(\bold s,\bold t)]=0.
$$
\end{definition}  

\begin{definition} A family $\nabla(\bold t)$ of multiplicative 
$\mathbbm{q}$-difference connections is {\bf periodic} if there exist operators $A_j\in GL(V)(\bold k(\bold t)(X))$, $1\le j\le m$ such that 
$$
\nabla(\mathbbm{q}^{\bold e_j}\bold t)\circ A_j(\bold t)=A_j(\bold t)\circ \nabla(\bold t),\ 1\le j\le m.
$$
\end{definition}  

It is clear that the $p$-curvature $C(\bold s,\bold t)$ of a periodic family 
of additive difference connections is conjugate, for generic $\bold s,\bold t$, to $C(\bold s+\bold e_j,\bold t)$, $1\le j\le n$, so the coefficients 
of the characteristic polynomial of any Laurent polynomial of $C_i^{\pm 1}$ are functions of $s_j-s_j^p$, $1\le j\le n$. Likewise, the existence of the operators $a_j$ for $n+1\le j\le n+m$ imply that these coefficients 
have zero derivatives with respect to $t_j$, so they are functions 
of $t_j^p$, $n+1\le j\le n+m$. Similarly, the $p$-curvature $C(\bold t)$ of a periodic family of multiplicative $\mathbbm{q}$-difference connections is conjugate, for generic $\bold t$, to $C(\mathbbm{q}^{\bold e_j}\bold t)$, $1\le j\le n$, so the coefficients of the characteristic polynomial of any Laurent polynomial of $C_i^{\pm 1}$ are functions of $t_j^p$, $1\le j\le n$.
 
Many examples of periodic families come from enumerative geometry. Namely, periodic families of additive difference connections arise as bispectrally dual equations to quantum connections on quantum cohomology of conical symplectic resolutions $\mathcal X$ (i.e., equations defined by shift operators), see \cite{MO}. Namely, in this case the multiplicative parameters $t_j$ are the coordinates on the torus $H^2(\mathcal X,\Bbb C^\times)$ (and the operators $t_j\partial_{t_j}-a_j$ are components of the quantum connection) and the additive parameter $\hbar$ is the equivariant parameter for the copy of $\Bbb C^\times$ dilating the cone $\overline{\mathcal X}$ resolved by $\mathcal X$. Similarly, periodic families of multiplicative difference connections arise as quantum connections on the quantum K-theory of such resolutions, as well as as bispectrally dual equations to them (which in many interesting cases are quantum connections on the quantum K-theory of mirror dual varieties), see \cite{OS}. In all these cases, the operators $A_i$ arise as geometric shift operators in quantum cohomology of K-theory (cf. \cite{EV}, Section 3.5). Examples of this type are considered in \cite{S,KS} and their $p$-curvature 
is treated in \cite{KS}.

\subsection{Periodic families and connection matrices} 

In \cite{EV} we showed that for families of regular differential connections over $\Bbb C$, a useful criterion for periodicity is the periodicity of the monodromy representation. We would now like to extend this to difference connections. For simplicity we will focus on 1-dimensional connections with one parameter $q$, and consider only multiplicative difference connections; a similar extension can be made 
in higher dimensions, for multiple parameters, and in the additive case.   

It is well known that for regular difference equations, the role of monodromy is played by the connection matrices, so let us recall 
their theory (see e.g. \cite{EFK}, Subsection 12.1 for a review). 

Let $\hbar\in \Bbb C$, ${\rm Re}\hbar>0$, and let 
$\mathbbm{q}:=e^{-\hbar}$. For a number or an operator $M$, 
$\mathbbm{q}^M$ will stand for $e^{-\hbar M}$. 
Consider the $\mathbbm{q}$-difference equation 
\begin{equation}\label{dieq}
F(\mathbbm{q}z)=B(z)F(z), 
\end{equation}
where $B\in GL_N(\Bbb C(z))$. One says that equation \eqref{dieq} is {\bf regular} if there exist limits $B_0=\mathbbm{q}^{M_0},B_\infty=\mathbbm{q}^{M_\infty}$ of 
$B(z)$ as $z\to 0,\infty$, for some matrices $M_0,M_\infty$. Let us additionally assume that the ratios of 
eigenvalues of $B_0$ are not equal to a nonzero integer power of $\mathbbm{q}$, and the same for $B_\infty$. 

In this case there exist unique matrix solutions of \eqref{dieq}
$$
F_0(z)=F_{0*}(z)z^{M_0},\ F_\infty(z)=F_{\infty*}(z^{-1})z^{M_\infty},
$$
where $F_{0*},F_{\infty*}$ are meromorphic on $\Bbb C$ with $F_{0*}(0)=F_{\infty*}(\infty)=1$, called {\bf asymptotic solutions} (so the functions $F_0,F_\infty$ may branch around $0$). The {\bf connection matrix} of \eqref{dieq} is then defined by 
$$
\bold C(z):=F_\infty(z)^{-1}F_0(z).
$$
It is easy to see that 
$$
\bold C(z)=\bold C(\mathbbm{q}z).
$$
Thus if $M_0,M_\infty$ are diagonalizable
then the matrix entries of $\bold C(z)$, as functions of $z$,  
express via theta-functions $\theta(z,\mathbbm{q})$. 
 
Now consider a family of $\mathbbm{q}$-difference connections
$\nabla(\bold t)=\bold T^{-1}B(\bold t,z)$, where $(\bold Tf)(z)=f(\mathbbm{q}z)$, and let 
$\bold C(\bold t,z)$ be the connection matrix of the corresponding difference equation $F(\mathbbm{q}z)=B(\bold t,z)F(z)$. Recall that this family is periodic if the connection $\nabla(\mathbbm{q}^{\bold e_i}\bold t)$ is gauge equivalent to 
$\nabla(\bold t)$ (for generic parameters). 

Denote by $Z_0,Z_\infty$ the centralizers of $M_0,M_\infty$ in $GL_N(\Bbb C)$. 

\begin{theorem}\label{perio} Suppose that for all $1\le i\le m$ there exist $g_i(\bold t)\in Z_0$, $h_i(\bold t)\in Z_\infty$ such that $\bold C(\mathbbm{q}^{\bold e_i}\bold t,z)= h_i(\bold t)\bold C(\bold t,z)g_i(\bold t)^{-1}$. Then the family $\nabla(\bold t)$ is periodic. 
\end{theorem} 

\begin{proof} Let $F_0(\bold t,z)$, $F_\infty(\bold t,z)$ be the corresponding 
asymptotic solutions. Then  
$$
F_0(\bold t,z)=F_\infty(\bold t,z)\bold C(\bold t,z),\ F_0(\mathbbm{q}^{\bold e_i}\bold t,z)g_i(\bold t)=F_\infty(\mathbbm{q}^{\bold e_i}\bold t,z)h_i(\bold t)\bold C(\bold t,z). 
$$
Thus 
$$
F_0(\mathbbm{q}^{\bold e_i}\bold t,z)g_i(\bold t)F_0(\bold t,z)^{-1}=F_\infty(\mathbbm{q}^{\bold e_i}\bold t,z)h_i(\bold t) F_\infty(\bold t,z)^{-1}.
$$
But the left hand side is meromorphic on $\Bbb C$, while 
the right hand side is meromorphic on $\Bbb C\Bbb P^1\setminus 0$. 
Thus both sides are meromorphic on $\Bbb C\Bbb P^1$, hence rational. 
So there exists $A_i(\bold t)\in GL_N(\Bbb C(\bold t,z))$ such that 
$$
F_0(\bold t,z)=A_i(\bold t,z)F_0(\mathbbm{q}^{\bold e_i}\bold t,z)g_i(\bold t). 
$$
Hence $A_i(\bold t,z)$ conjugates 
$\nabla(\bold t)$ into 
$\nabla(\mathbbm{q}^{\bold e_i}\bold t)$.
Using Chevalley's elimination of quantifiers as in \cite{EV}, Subsection 4.1, we can choose $A_i$ to depend rationally of $\bold t$, as claimed.  
\end{proof} 

Theorem \ref{perio} allows one to show that many naturally occurring families of ($\mathbbm{q}$-)difference connections are periodic. For example, this applies to the $q$-hypergeometric family (see \cite{EFK}, Subsection 12.2)\footnote{In this case, the groups $Z_0,Z_\infty$ are both the torus $(\Bbb C^\times)^2$ and the entries of the ``gauge transformations" $g_i,h_i$ are comprised of the $\Gamma_{\mathbbm{q}}$-factors of formula (12.21) in \cite{EFK}.}
and more generally to the quantum KZ equations 
for $GL_N$, as their connection matrices are known to be given, up to a gauge transformation, by Felder's elliptic dynamical R-matrices (see \cite{TV1} for $N=2$ and \cite{E}, Subsections 3.10, 3.11 and \cite{EM} for general $N$). 
Another example is Cherednik's  quantum KZ equations for root systems of type $C$, whose connection matrices are computed in \cite{St} in terms of theta functions. 

\subsection{Example of computation of the spectrum of the $p$-curvature}

In this subsection we give the simplest nontrivial example of computation of the spectrum of the $p$-curvature of a periodic family,\footnote{This example is a special case of the more general results of \cite{KS}.} along the lines of Theorem \ref{main1}. Namely, it describes 
the spectrum of the $p$-curvature for the quantum KZ equation for $GL_N$ valued in the tensor product of two vector representations, as well as for its additive analog. 

Let 
$$
R(q,z):=\begin{pmatrix} \frac{qz-1}{z-1} & \frac{(q-1)z}{z-1}\\
\frac{q-1}{z-1} & \frac{qz-1}{z-1}\end{pmatrix}=1+
(q-1)\begin{pmatrix} \frac{z}{z-1} & \frac{z}{z-1}\\
\frac{1}{z-1} & \frac{z}{z-1}\end{pmatrix},
$$
and let $t\in \Bbb C^\times$. Consider the 2-parameter family of 
$\mathbbm{q}$-difference connections 
\begin{equation}\label{difcon}
\nabla(q,t)=\bold T^{-1}R(q,z)\begin{pmatrix} t& 0\\ 0 & t^{-1}\end{pmatrix}.
\end{equation} 
This is the simplest instance of the quantum KZ connection, which is also 
a $q$-hypergeometric connection. The computation of connection matrices for the corresponding $\mathbbm{q}$-difference equation 
(see e.g. \cite{EFK}, Subsections 12.2, 12.3) along with Theorem \ref{perio} implies that $\nabla(q,t)$ is a periodic family for sufficiently large $p$. 
Also, if $\mathbbm{q}$ is a primitive $p$-th root of $1$ 
then the $p$-curvature of $\nabla(q,t)$ is defined by the formula
$$
C(q,t,z)=\prod_{j=p-1}^{0}\left(R(q,\mathbbm{q}^jz)\begin{pmatrix} t & 0\\ 0 & t^{-1}\end{pmatrix}
\right).
$$

\begin{theorem}\label{isospee} For sufficiently large $p$, $C(q,t,z)$ is isospectral to $R(q^p,z^p)\begin{pmatrix} t^{p} & 0\\ 0 & t^{-p}\end{pmatrix}$. 
\end{theorem} 

\begin{proof} 
Let $p$ be a positive integer and $\bold k$ a field of characteristic not dividing $p$. Let $A_j$, $1\le j\le p$, be square matrices of size $N$ over $\bold k$ and $A:=\sum_{j=1}^p A_j$. 

\begin{lemma}\label{leee1} Let $q$ be a variable and suppose that the characteristic polynomial 
of the matrix 
$$
M:=\prod_{j=1}^p (1+(q-1)A_j)
$$
depends only on $q^p$. Then $M$ is isospectral to $1+\frac{q^p-1}{p}A$. 
\end{lemma}  

\begin{proof} We have $M-1=(q-1)A+O((q-1)^2)$, $q\to 1$. Thus 
$$
{\rm Tr}\wedge^m (M-1)=(q-1)^m{\rm Tr}\wedge^m A+O((q-1)^{m+1}),\ q\to 1. 
$$
On the other hand, ${\rm Tr}\wedge^m (M-1)$ is a polynomial in 
$q$ of degree $pm$ which is, in fact, a polynomial of $q^p$ (thus, of degree $m$). 
Thus ${\rm Tr}\wedge^m (M-1)=\frac{(q^p-1)^m}{p^m}{\rm Tr}\wedge^m A$, so $M-1$ is isospectral to $\frac{q^p-1}{p}A$, as desired. 
\end{proof} 

\begin{lemma}\label{propo} Let $\bold t:={\rm diag}(t_1,...,t_N)$ and $M(\bold t):=\prod_{j=1}^p ((1+(q-1)A_j)\bold t)$. If the characteristic polynomial of $M(\bold t)$ depends only on $\bold t^p$ and $q^p$ then $M(\bold t)$ is isospectral to $(1+\frac{q^p-1}{p}A)\bold t^p$. 
\end{lemma}  

\begin{proof} 
By assumption, 
$$
{\rm Tr}\wedge^m M(\bold t)=\sum_{i_1<...<i_m}C_{i_1...i_m}t_{i_1}^p...t_{i_m}^p.
$$
The coefficient $C_{i_1...i_m}$ can be computed by setting $t_j$ to $1$ when 
$j=i_k$ for some $k$ and $t_j=0$ otherwise. In this case 
$$
C_{i_1...i_m}={\rm Tr}\wedge^m M(\bold t)=\det \prod_{j=1}^p (1+(q-1)P_{i_1...i_m}A_jP_{i_1...i_m}),
$$
where $P_{i_1,...,i_m}$ is the projector to the subspace spanned by $\bold e_{i_1},...,\bold e_{i_m}$ 
which kills the other standard basis vectors, and the determinant is taken over the image of this projector. So by Lemma \ref{leee1} we get 
$$
C_{i_1...i_m}=\det(1+\tfrac{q^p-1}{p}P_{i_1...i_m}AP_{i_1...i_m}).
$$
This implies that for general $\bold t$
$$
{\rm Tr}\wedge^m M(\bold t)={\rm Tr}\wedge^m ((1+\tfrac{q^p-1}{p}A)\bold t^p),
$$
as claimed.
\end{proof} 

Now we apply Lemma \ref{propo} to the family \eqref{difcon}, which we can do since this family is periodic.
If $\mathbbm{q}$ is a primitive $p$-th root of $1$ 
then 
$$
\sum_{j=0}^{p-1}\frac{1}{\mathbbm{q}^jz-1}=\frac{p}{z^p-1},\
\sum_{j=0}^{p-1}\frac{\mathbbm{q}^jz}{\mathbbm{q}^jz-1}=\frac{pz^{p}}{z^p-1}.
$$
Thus it follows from Lemma \ref{propo} that the $p$-curvature 
$C(q,t,z)$ is isospectral to 
$$
\left(1+(q^p-1)\begin{pmatrix} \frac{z^p}{z^p-1} & \frac{z^p}{z^p-1}\\
\frac{1}{z^p-1} & \frac{z^p}{z^p-1}\end{pmatrix}\right)\begin{pmatrix} t^{p} & 0\\ 0 & t^{-p}\end{pmatrix}=R(q^p,z^p)\begin{pmatrix} t^{p} & 0\\ 0 & t^{-p}\end{pmatrix}.
$$
\end{proof} 

\begin{remark} The shift operator in this case can be constructed explicitly and is defined for all $p$, which allows us to conclude that Theorem \ref{isospee} in fact holds for all $p$. 
\end{remark} 

The additive limit of the connection \eqref{difcon} is obtained by writing $\mathbbm{q}=1+\varepsilon$, 
$z=1+u\varepsilon$, $q=1+s\varepsilon$, and sending $\varepsilon$ to $0$. Then 
$$
R(q,z)\to R_{\rm add}(s,u):=1+
\frac{s}{u}\begin{pmatrix} 1 & 1\\
1 & 1\end{pmatrix},
$$
so the periodic family $\nabla(q,t)$ degenerates to the mixed-periodic family
\begin{equation}\label{difcon1}
\nabla_{\rm add}(s,t)=T^{-1}R_{\rm add}(s,u)\begin{pmatrix} t& 0\\ 0 & t^{-1}\end{pmatrix},
\end{equation} 
where $(Tf)(u)=f(u+1)$ (it is periodic in $s$ and infinitesimally periodic in $t$). Thus by taking the additive limit in Theorem \ref{isospee}, we obtain the following corollary.\footnote{Namely, recall that if $\mathbbm{q}$ is a primitive $p$-th root of unity in 
the cyclotomic ring $\Bbb Z_p[\mathbbm{q}]$ then 
$\mathbbm{q}=1+\varepsilon$ where $\varepsilon$ is a uniformizer of this ring with $p$-adic valuation $v_p(\varepsilon)=\frac{1}{p-1}$, and such that $\varepsilon^p=-p\varepsilon+O(\varepsilon^{p+1})$. Thus 
$q^p=(1+s\varepsilon)^p=1+(s^p-s)\varepsilon^p+O(\varepsilon^{p+1})$, 
$z^p=(1+u\varepsilon)^p=1+(u^p-u)\varepsilon^p+O(\varepsilon^{p+1})$. 
Hence the reduction of $\frac{q^{p}-1}{z^{p}-1}$ from the cyclotomic ring $\Bbb Z_p[\mathbbm{q}]$ to $\Bbb F_p$ equals $\frac{s^p-s}{u^p-u}$.}

\begin{corollary} The $p$-curvature $C_{\rm add}(s,t,u)$ of the connection \eqref{difcon1} reduced to characteristic $p$ is isospectral to $R_{\rm add}(s^p-s,u^p-u)\begin{pmatrix} t^p& 0\\ 0 & t^{-p}\end{pmatrix}$. 
\end{corollary} 

\begin{remark}\label{ks} Most of the (mixed-)periodic pencils of flat connections considered in Section 5 of \cite{EV} and Section 4 of this paper admit 
difference and $\mathbbm{q}$-difference analogs. Unfortunately, 
the technique of this subsection is not powerful enough to prove 
a counterpart of Theorem \ref{isospee} for most of such ($\mathbbm{q}$-)difference connections, as it uses that the connection form $B$ is linear in $q$ and $\bold t$. Yet we expect that the natural analog of Theorem \ref{isospee} is true for all such connections (i.e., 
the $p$-curvature is isospectral to the suitably defined Frobenius twist of the connection form).\footnote{For instance, in the case of $\mathbbm{q}$-difference Dunkl-Cherednik connections, we expect that there is a \linebreak $q$-deformation of the theory of 
Subsection \ref{rca}, based on the representation theory of Cherednik's double affine Hecke algebras (DAHA) at roots of unity, \cite{Ch}. In the $q$-Toda limit, 
DAHA should be replaced by nil-DAHA (\cite{Ch2}).} In fact, recently a method to obtain these results for a large class of connections (qKZ equations attached to Nakajima varieties) using their representation as generalized $\mathbbm{q}$-Gauss-Manin connections arising from enumerative geometry was proposed in \cite{KS}. This method moreover provides an explicit base change matrix identifying the Frobenius twisted connection form with the $p$-curvature.
\end{remark}

\end{document}